\documentclass[12pt]{amsart}
\usepackage{amscd,amsmath,amsthm,amssymb,enumerate, comment}
\usepackage[left]{lineno}
\usepackage{color}
\usepackage{pstricks}
\usepackage[utf8]{inputenc}
\usepackage{epstopdf}
\usepackage{graphicx}

\usepackage[all]{xy}

\newpsstyle{fatline}{linewidth=1.5pt}
\newpsstyle{fyp}{fillstyle=solid,fillcolor=verylight}
\definecolor{verylight}{gray}{0.97}
\definecolor{light}{gray}{0.9}
\definecolor{medium}{gray}{0.85}
\definecolor{dark}{gray}{0.6}

 \def\NZQ{\mathbb}               

 \def\ZZ{{\NZQ Z}}
 \def\RR{{\NZQ R}}

 \def\frk{\mathfrak}               
 
 \def\pp{{\frk p}}

 \def\mm{{\frk m}}

 \def\Ic{{\mathcal I}}

 \def\G{{\mathcal G}}

\def\Bc{{\mathcal B}}

 \def\ab{{\mathbf a}}
 
 \def\xb{{\mathbf x}}

 \def\eb{{\mathbf e}}
 \def\pb{{\mathbf p}}
 \def\qb{{\mathbf q}}

\def\ab{{\mathbf a}}
\def\fb{{\mathbf f}}

 \def\Soc{{\mathbf Soc}}

 \def\opn#1#2{\def#1{\operatorname{#2}}} 
 \opn\chara{char} \opn\length{\ell} \opn\pd{pd} \opn\rk{rk}
 \opn\projdim{proj\,dim} \opn\injdim{inj\,dim} \opn\rank{rank}
 \opn\depth{depth} \opn\grade{grade} \opn\height{height}
 \opn\bigheight{bigheight}
 \opn\embdim{emb\,dim} \opn\codim{codim}
 
 \opn\Tr{Tr} \opn\bigrank{big\,rank}
 \opn\superheight{superheight}\opn\lcm{lcm}
 \opn\trdeg{tr\,deg}
 \opn\reg{reg} \opn\lreg{lreg} \opn\ini{in} \opn\lpd{lpd}
 \opn\size{size} \opn\sdepth{sdepth}
 \opn\link{link}\opn\fdepth{fdepth}\opn\lex{lex}
 \opn\tr{tr}
 \opn\type{type}
 \opn\gap{gap}
 \opn\arithdeg{arith-deg}
 \opn\Deg{Deg}
 \opn\sat{sat}
 \opn\mat{mat}
 \opn\Mat{Mat}
\opn\mdeg{mdeg}
\opn\pred{pred}
\opn\succ{succ}
 \opn\div{div} \opn\Div{Div} \opn\cl{cl} \opn\Cl{Cl}
 \opn\Spec{Spec} \opn\Supp{Supp} \opn\supp{supp} \opn\Sing{Sing}
 \opn\Ass{Ass} \opn\Min{Min}\opn\Mon{Mon}
 \opn\Ann{Ann} \opn\Rad{Rad} \opn\Soc{Soc}
 \opn\Im{Im} \opn\Ker{Ker} \opn\Coker{Coker} \opn\Am{Am}
 \opn\Hom{Hom} \opn\Tor{Tor} \opn\Ext{Ext} \opn\End{End}
 \opn\Aut{Aut} \opn\id{id}
 
 \opn\nat{nat}
 \opn\pff{pf}
 \opn\Pf{Pf} \opn\GL{GL} \opn\SL{SL} \opn\mod{mod} \opn\ord{ord}
 \opn\Gin{Gin} \opn\Hilb{Hilb}\opn\sort{sort}
 \opn\PF{PF}\opn\Ap{Ap}
 \opn\mult{mult}
 \opn\bight{bight}
 \opn\aff{aff}
 \opn\relint{relint} \opn\st{st}
 \opn\lk{lk} \opn\cn{cn} \opn\core{core} \opn\vol{vol}  \opn\inp{inp} \opn\nilpot{nilpot}
 \opn\link{link} \opn\star{star}\opn\lex{lex}\opn\set{set}
 \opn\width{wd}
 \opn\Fr{F}
 \opn\QF{QF}
 \opn\G{G}
 \opn\type{type}\opn\res{res}
 \opn\conv{conv}
 \opn\Shad{Shad}
 \opn\gr{gr}
 
 \def\pot#1#2{#1[\kern-0.28ex[#2]\kern-0.28ex]}

 \opn\dirlim{\underrightarrow{\lim}}
 \opn\inivlim{\underleftarrow{\lim}}
 \let\sect=\cap
 
 \let\tensor=\otimes
 \let\iso=\cong

 \let\to=\rightarrow
 
 \def\Implies{\ifmmode\Longrightarrow \else
         \unskip${}\Longrightarrow{}$\ignorespaces\fi}
 \def\implies{\ifmmode\Rightarrow \else
         \unskip${}\Rightarrow{}$\ignorespaces\fi}
 \def\iff{\ifmmode\Longleftrightarrow \else
         \unskip${}\Longleftrightarrow{}$\ignorespaces\fi}

 \let\:=\colon
\theoremstyle{plain}
 \newtheorem{Theorem}{Theorem}[section]
 \newtheorem{Lemma}[Theorem]{Lemma}
 \newtheorem{Corollary}[Theorem]{Corollary}
 \newtheorem{Proposition}[Theorem]{Proposition}
 \newtheorem{Claim}[Theorem]{Claim}

\theoremstyle{definition}

 \newtheorem{Example}[Theorem]{Example}
 
 \newtheorem{Definition}[Theorem]{Definition}
 \newtheorem{Problem}[Theorem]{Problem}

 \newtheorem{Question}[Theorem]{Question}
 
\newtheorem*{acknowledgments}{Acknowledgments}

 \let\epsilon\varepsilon
 \let\kappa=\varkappa
 \def\qed{\ifhmode\textqed\fi
       \ifmmode\ifinner\quad\qedsymbol\else\dispqed\fi\fi}
 \def\textqed{\unskip\nobreak\penalty50
        \hskip2em\hbox{}\nobreak\hfil\qedsymbol
        \parfillskip=0pt \finalhyphendemerits=0}
 \def\dispqed{\rlap{\qquad\qedsymbol}}
 
 \opn\dis{dis}
 \def\pnt{{\raise0.5mm\hbox{\large\bf.}}}
 
 \opn\Lex{Lex}

 \def\fkm{{\frk m}}
\def\fkp{{\frk p}}
\def\fkq{{\frk q}}

\def\calR{{\mathcal{R}}}

\def\ol{\overline}

\title{Graded Bourbaki ideals of graded modules}

\author{J\"{u}rgen Herzog}
\address{J\"urgen Herzog: Fachbereich Mathematik, Universit\"at Duisburg-Essen, Fakult\"at f\"ur Mathematik, 45117 Essen, Germany}
\email{juergen.herzog@uni-essen.de}

\author{Shinya Kumashiro}
\address{Shinya Kumashiro: Department of Mathematics and Informatics, Graduate School of Science and Engineering, Chiba University, Yayoi-cho 1-33, Inage-ku, Chiba, 263-8522, Japan}
\email{caxa2602@chiba-u.jp}

\author{Dumitru I. Stamate}
\address{Dumitru I. Stamate: Faculty of Mathematics and computer science, university of Bucharest, Str. Academiei 14, Bucharest - 010014, Romania}
\email{dumitru.stamate@fmi.unibuc.ro}

\thanks{The second author was supported by JSPS KAKENHI Grant Number JP19J10579 and JSPS Overseas Challenge Program for Young Researchers.}
\thanks{The third author was partly supported by the University of Bucharest, Faculty of Mathematics and Computer Science through the 2019 Mobility Fund.}

\begin{document}

\begin{abstract}
In this paper  we study  graded Bourbaki ideals. It is a well-known fact that for  torsionfree modules over  Noetherian normal domains,  Bourbaki sequences exist. We  give  criteria in terms  of certain attached matrices for a  homomorphism of modules to induce a Bourbaki sequence. Special attention is given to graded Bourbaki sequences.  In the second  part of the paper, we apply these  results to the Koszul cycles of the residue class field and determine particular Bourbaki ideals explicitly. We also obtain in a special case  the relationship  between the structure of the Rees algebra of a Koszul cycle and the Rees algebra of its Bourbaki ideal.
\end{abstract}

\maketitle



\section{Introduction}\label{section1}

The purpose of  this paper is to study  Bourbaki sequences and Bourbaki ideals. Throughout this section let $R$ be a commutative Noetherian ring and $M$ a finitely generated $R$-module. Then a {\it Bourbaki sequence} of $M$ is a short exact sequence
\begin{equation}
\label{bbb}
0\to F\to M \to I \to 0
\end{equation}
of $R$-modules, where $F$ is a free $R$-module and $I$ is an ideal of $R$. $I$ is called a {\it Bourbaki ideal} of $M$.
As a fundamental result, a Bourbaki sequence of $M$ always exists if $R$ is a normal domain and $M$ is a finitely generated torsionfree $R$-module (see \cite[Chapter VII, Section 4, 9. Theorem 6.]{B}). If $R$ is a standard graded normal domain over an infinite field, then a graded Bourbaki sequence of $M$ also exists (Theorem \ref{basic2.1}, see also \cite[Corollary 2.4]{Ku2}).
One of the advantages  of Bourbaki's theorem is the fact, that,  by passing to a Bourbaki sequence,  many properties of a module are inherited by those  of its Bourbaki ideals. One can find applications of Bourbaki's theorem, for instance, to the vanishing of cohomologies, the study of the maximal Cohen-Macaulay modules over hypersurface rings,  the Hilbert functions, and the Rees algebras of modules  (\cite{Au, HKu, HTZ, Ku2, SUV2, W}).

On the other hand, even though we know about the existence of a Bourbaki sequence, it is not easy to construct one explicitly.  Actually, for a given homomorphism of modules, it is still difficult to check whether the map induces a Bourbaki sequence.

\begin{Problem}\label{problem}
Let $R$ be a Noetherian ring and $M$ a finitely generated $R$-module. Let $F$ be a finitely generated free $R$-module. Then, for a given homomorphism  $\varphi: F\to M$ of modules, when does the sequence
\[
0\to F\xrightarrow{\varphi} M \to \Coker(\varphi) \to 0
\]
provide a Bourbaki sequence, and if this is the case,  how to compute the corresponding  Bourbaki ideal?
\end{Problem}

Now let us explain how we organized this paper.
In Section \ref{section2} we prepare general propositions to study Problem \ref{problem}, and discuss the existence of graded Bourbaki sequences.
In Section \ref{section3} we introduce an  invariant, the {\it Bourbaki number}, obtained from a graded Bourbaki sequence over the polynomial ring. The Bourbaki number is an integer which only depends on the degree of the generators and invariants of $M$. It will be useful to find in Section \ref{section4} graded Bourbaki ideals for Koszul cycles.

In Section \ref{section3.5} we solve Problem \ref{problem} under some additional conditions, as described in the following theorem.

\begin{Theorem}{\rm (cf. Theorem \ref{a4.2} and \ref{a4.3})}
Let $R$ be a normal domain of dimension $\geq 2$ and $M$ a finitely generated torsionfree $R$-module of rank  $r>1$. Let  $\varphi: R^{r-1} \to M$ be an $R$-module homomorphism.
\begin{enumerate}[{\rm (a)}]
\item Suppose that $M$ is reflexive and take an exact sequence $0\to M \xrightarrow{\iota} F \to X \to 0$ so that $F$ is a  free $R$-module and $X$ is a finitely generated torsionfree $R$-module. Then the following conditions are equivalent:
\begin{enumerate}
\item[{\rm (i)}] $0\to R^{r-1}\xrightarrow{\varphi} M \to \Coker(\varphi) \to 0$ is a Bourbaki sequence;
\item[{\rm (ii)}] $\height (I_{r-1}(\iota\circ \varphi)) \geq 2$.
\end{enumerate}

Here, $I_t(\alpha)$ denotes the ideal of $t$-minors of a matrix representing $\alpha$, where $\alpha$ is a module homomorphism between finitely generated free $R$-modules.

\item Suppose that $\varphi$ is an injective map and $\projdim_R M <\infty$. Let
\begin{align*}
R^{\beta_1} \xrightarrow{\psi} R^{\beta_0} \to \Coker (\varphi) \to 0
\end{align*}
be a presentation for $\Coker (\varphi)$. Then the following conditions are equivalent:
\begin{enumerate}
\item[{\rm (i)}] $0\to R^{r-1}\xrightarrow{\varphi} M \to \Coker(\varphi) \to 0$ is a Bourbaki sequence;
\item[{\rm (ii)}] $\height(I_{\beta_0-r+1}(\psi)) \geq 2$.
\end{enumerate}
\end{enumerate}
\end{Theorem}

As an application of the above theorem, we will illustrate the ubiquity of graded Bourbaki sequences (Theorem \ref{b4.5}).
Furthermore, we also give a method to compute a Bourbaki ideal for a given Bourbaki sequence (Theorem \ref{a4.4}).

Let $S=K[x_1, \dots, x_n]$ be the polynomial ring. In Section \ref{section4} we apply the previous results to the Koszul cycles $Z_i$ of the residue class field $K$,  and determine particular  Bourbaki ideals explicitly in the cases $i=2,\ n-2,\ n-1$. For $i=n-1$ and $i=n-2$, we can choose multigraded Bourbaki sequences.  Hence the corresponding Bourbaki ideals are monomial ideals (Proposition \ref{a5.2} and \ref{c5.3}). On the other hand, as shown in Theorem \ref{a5.4}, multigraded  Bourbaki sequences do not exist for $1<i< n-2$ when $n\gg 0$ or $n\leq 6$. We  expect it  is also the case for all $n> 6$.

In the last part of this paper we show that our Bourbaki ideal for $Z_{n-2}$ has the property that its Rees  algebra is normal and Cohen--Macaulay. Moreover, for $n$ even it is Gorenstein and it is of  Cohen-Macaulay type 2 if $n$ is odd. The same properties are known for the Rees algebra of $Z_{n-2}$, see \cite[Theorem 3.1 and Theorem 3.4]{SUV}. 

Let us fix our notation throughout this paper. In what follows, let $R$ be a commutative Noetherian ring.
Let $\varphi: F\to G$ be an $R$-module homomorphism between finitely generated free $R$-module $F$ and $G$. We then denote $I_{t} (\varphi)$ the ideal of $R$ generated by the $t$-minors of a matrix representing $\varphi$.
$\mathrm{Q}(R)$ denotes the total ring of fraction of $R$ and the functor $(-)^*$ denotes the $R$-dual. For a finitely generated $R$-module $M$, we say that $M$ is {\it torsionfree} (resp. {\it reflexive}) if the canonical map $M\to \mathrm{Q}(R)\otimes_R M$ is injective (resp. the canonical map $M\to M^{**}$ is bijective).

If $R=\bigoplus_{n\geq 0} R_n$ is a graded Noetherian ring over a field $K=R_0$ and $M$ is a finitely generated graded $R$-module, 
$t_0(M)$ denotes the maximum degree of an element in a minimal homogeneous system of generators of  $M$.

\begin{acknowledgments}
We gratefully acknowledge the use of the computer algebra software CoCoA (\cite{Cocoa}) and Singular (\cite{Sing}) for our computations. The authors would like to thank Ernst Kunz and Bernd Ulrich for discussions around Lemma \ref{b4.3}.
  We thank an anonymous referee for suggestions which clarified the presentation. 

This paper was written while the second author visited Essen for three months and Bucharest for two weeks.
During the visits, Professor Herzog and Professor Stamate were extremely kind to him every day and he would like to express sincere gratitude for their hospitality. 
The third author is grateful to Professor Herzog and the Department of Mathematics of the  Universit\"at Duisburg-Essen  for being excellent hosts, one more time.

\end{acknowledgments}

\section{Preliminaries}\label{section2}

Let $R$ be a Noetherian ring and $M$ a finitely generated $R$-module.
As a fundamental result, a Bourbaki sequence of $M$ always exists if $R$ is a normal domain and $M$ is a finitely generated torsionfree $R$-module (see \cite[Chapter VII, Section 4, 9. Theorem 6.]{B}).
We also have a graded version of Bourbaki sequences.

\begin{Theorem}\label{basic2.1}
Let $R=\bigoplus_{n\geq 0} R_n$ be a standard graded Noetherian normal domain where $R_0$ is an infinite field and $\dim R\ge 2$. Let $M=\bigoplus_{n\in \mathbb{Z}} M_n$ be a finitely generated torsionfree graded $R$-module of rank  $r>1$.

Then for any integer $k\ge t_0(M)$, there exists a graded Bourbaki sequence
\begin{align}\label{Bseq}
0 \to R(-k)^{r-1} \to M \to I(m) \to 0
\end{align}
of $M$, for some  integer $m$  and $I$ is a graded ideal of $R$.

 Moreover, if $R$ is a factorial ring, then there exists a Bourbaki sequence as in \eqref{Bseq} with $\grade (I)\geq 2$.
 \end{Theorem}

\begin{proof}
 By lack of good reference we outline the proof.
Since $R$ is standard graded and $k\ge t_0(M)$, it follows that  
 $M_{\ge k}$ is generated in degree $k$.

By  \cite[Corollary~2.4]{Ku2} we have an exact  sequence
\[
0 \to R(-k)^{r-1} \to M_{\ge k} \to N' \to 0,
\]
where $N'$ is a rank $1$ torsionfree $R$-module. From this  we construct a graded Bourbaki sequence of $M$.  Indeed, consider the following commutative diagram
\[
\xymatrix{
&&0 \ar[d]&0 \ar[d]&\\
0 \ar[r] & R(-k)^{r-1}\ar[r] \ar@{=}[d] &M_{\ge k} \ar[r] \ar[d] & N' \ar[r] \ar[d] & 0\\
0 \ar[r] & R(-k)^{r-1}\ar[r]  & M \ar[r] \ar[d] & N \ar[r] \ar[d] & 0 \\
&&M/M_{\ge k}\ar@{=}[r] \ar[d]&M/M_{\ge k} \ar[d]&\\
&&0&0&,\\
}
\]
where $N$ denotes the cokernel of the composition $R(-k)^{r-1}\to M_{\ge k} \to M$.
Since $M/M_{\ge k}$ has finite length, $N$ has rank one and $\Ass (N)\subseteq \{0, \frk{m}\}$, where $\fkm=R_{>0}$ is the  graded maximal ideal in $R$. If $\frk{m}\in \Ass (N)$, since $\depth R_{\frk{m}}\ge 2$, it follows that $\depth_{R_\frk{m}} M_{\frk{m}}=0$ and $\fkm \in \Ass (M)$. This is a contradiction for the torsionfreeness of $M$. Whence $\Ass (N)=\{0\}$. Therefore, $N$ is torsionfree of rank~1. Let $S$  be the multiplicatively closed set of non-zero homogeneous elements in $R$. It follows that $N\to S^{-1}R\tensor_RN$ is injective and $S^{-1}R\tensor_RN \iso S^{-1}R$. This implies that $N\iso I(m)$ as a graded $R$-module, where $I\subset R$ is a graded ideal and $m$ is a suitable integer.


Now assume in addition that $R$ is a factorial ring. If  $I$ is of grade $1$,  then $I=\alpha{\cdot}J$ for some graded ideal $J$ with
$\gcd (J)=1$. Therefore we may as well assume that $\gcd(I)=1$,  and hence  $\height(I) \ge 2$. Since $R$ is a factorial domain, it is a normal ring and satisfies Serre's condition $(S_2)$. Therefore,  $\grade (I)\ge 2$.
\end{proof}


Non-trivial Bourbaki sequences we only obtain when the grade of the Bourbaki ideal is $2$.  Indeed, we have

\begin{Lemma}\label{b2.2}
Let $R$ be a Noetherian ring and $M$ a finitely generated $R$-module with Bourbaki sequence  $0\to F \to M \to I \to 0$.
If $\grade(I)>2$, then $M\iso F\oplus I$.

In particular, if $M$ is not free and reflexive or $M$ is  an indecomposable module of rank $\ge 2$, then $\grade (I)=2$.
\end{Lemma}

\begin{proof}
If $\grade(I)>2$, then $\Ext_R^1 (I, R)\cong \Ext_R^2 (R/I, R)=0$. Whence, the Bourbaki sequence $0\to F \to M \to I \to 0$ of $M$ splits.
\end{proof}


\begin{Proposition}\label{b2.3}
Let $R$ be a Cohen-Macaulay normal domain and $M$ a finitely generated torsionfree $R$-module with $\projdim M<\infty$. Let
\[
0 \to F \to M \to I \to 0
\]
be a Bourbaki sequence of $M$ such that  $\grade(I)=2$. Then
\begin{align*}
&\{\frk{p}\in \Spec R \mid \projdim_{R_{\fkp}}M_{\fkp}\le 1\} \\
=&\{\frk{p}\in \Spec R \mid \text{$(R/I)_{\fkp}$ is  a Cohen-Macaulay ring or zero}\}.
\end{align*}

In particular, $R/I$ is Cohen-Macaulay on the punctured spectrum of $R$ if $M$ is locally free on the punctured spectrum of $R$.
\end{Proposition}

\begin{proof}
For $\frk{p}\in \Spec R$, let
\begin{equation}\label{localization}
0 \to F_\fkp \to M_{\fkp} \to IR_{\fkp} \to 0
\end{equation}
be the localization of the above Bourbaki sequence. Then $\projdim_{R_\fkp} M_\fkp \le 1$ if and only if $\projdim_{R_\fkp} IR_\fkp \le 1$ since $\Ext_{R_\fkp}^2(M_\fkp, X)\iso \Ext_{R_\fkp}^2(IR_\fkp, X)$ for all $R_\fkp$-module $X$ by (\ref{localization}).

On the other hand, since $\grade (I)=2$, it follows that $\projdim_{R_\fkp} IR_\fkp=0$ $\Leftrightarrow$ $I\not\subseteq \fkp$, and $\projdim_{R_\fkp} IR_\fkp=1$ $\Leftrightarrow$ $IR_\fkp$ is perfect  in the sense of \cite[Definition 1.4.15]{BH}, in other words, a Cohen-Macaulay ideal.
\end{proof}







\section{The Bourbaki number of  graded torsionfree modules}\label{section3}

In this section,  $S=K[x_1, x_2, \dots, x_n]$ is a  polynomial ring of dimension $n\ge 2$ over an infinite field $K$.
For any finitely generated graded $S$-module $M$ of positive dimension $s$, the Hilbert function of $M$, which is defined as
$H_{M}(t)=\dim_K M_t$ for all $t\in \ZZ$, eventually agrees with a polynomial function of degree $s-1$.  Thus we may write
\[
\dim_K M_t=\mathrm{e}_0(M) \binom{t+s-1}{s-1} - \mathrm{e}_1(M) \binom{t+s-2}{s-2} + \cdots +(-1)^{s-1} \mathrm{e}_{s-1}(M)
\]
  for all $t\gg 0$, see \cite[Theorem 4.1.3]{BH}.  The integers $\mathrm{e}_0(M), \mathrm{e}_1(M), \dots, \mathrm{e}_{s-1}(M)$ are called the {\it Hilbert coefficients} of $M$.

\begin{Theorem}\label{a2.1}
Let $M$ be  a finitely generated torsionfree graded $S$-module of rank $r>1$. For $k\geq t_0(M)$ let
\begin{equation}\label{a3.1.1}
0 \to S(-k)^{r-1} \to M \to I(m) \to 0
\end{equation}
be  a graded Bourbaki sequence of $M$ with $\grade(I)\geq 2$. Then
$m=k{\cdot}(r-1) - \mathrm{e}_1 (M)$.
\end{Theorem}

\begin{proof}
By using the additivity of the Hilbert function on the graded exact sequence
\[
0 \to S(-k)^{r-1} \to M \to S(m) \to (S/I)(m) \to 0,
\]
we get that for $t\gg 0$,
{\small
\begin{align*}
H_{S/I}(t+m) =& H_{S}(t+m) -H_{M}(t) + (r-1){\cdot}H_{S}(t-k) \\
  =& \binom{t+m+n-1}{n-1} - H_{M}(t) + (r-1){\cdot}\binom{t-k+n-1}{n-1}\\
  =& \binom{t+n-1}{n-1} +m\binom{t+n-2}{n-2} -\left\{ \mathrm{e}_0(M)\binom{t+n-1}{n-1} - \mathrm{e}_1 (M)\binom{t+n-2}{n-2}\right\} \\
  &  + (r-1)\left\{ \binom{t+n-1}{n-1}-k\binom{t+n-2}{n-2}\right\} \\& +(\text{a polynomial in $t$ of  degree $<n-2$})\\
  =& \left\{1-\mathrm{e}_0(M)+(r-1)\right\}\binom{t+n-1}{n-1} +\left\{m+\mathrm{e}_1(M)-k(r-1)\right\}\binom{t+n-2}{n-2}\\
  &  + (\text{a polynomial in $t$ of  degree $<n-2$}).
\end{align*}
}
Since the dimension of $S/I$ is at most $n-2$,   the degree of the polynomial $H_{S/I}(t+m)$ in $t$ is at most $n-3$. Hence $m+\mathrm{e}_1(M)-k(r-1)=0$, which gives the desired formula for $m$.
%
\end{proof}



Theorem~\ref{a2.1} states that for a given $M$, the integer $m$ in the Bourbaki sequence \eqref{a3.1.1}  does not depend on the embedding of $S(-k)^{r-1}$ into $M$, but only on $k$, under the not so restrictive assumption that $\grade(I) \geq 2$ (see also Theorem~\ref{basic2.1}).

\begin{Definition}
We say that the integer $m$ in the exact sequence \eqref{a3.1.1} is the {\it Bourbaki number} of $M$  with respect to $k$.
\end{Definition}

If a graded  free resolution of $M$ is  known, $\mathrm{e}_1 (M)$ can be computed as follows.

\begin{Proposition}\label{a3.3}
Let $M$ be a finitely  generated graded $S$-module with $\dim M=\dim S$, and let $0\to F_p \to \dots \to F_1 \to F_0 \to M \to 0$ be  any finite graded free resolution of $M$. If  $F_i=\bigoplus_{j\in \mathbb{Z}}S(-j)^{b_{ij}}$ for $0\le i \le p$, then
\[
\mathrm{e}_1 (M)=\sum_{i=0}^{p}\sum_{j\in \mathbb{Z}} (-1)^{i}{\cdot}j{\cdot}b_{ij}.
\]
\end{Proposition}

\begin{proof}
Because
\begin{align*}
H_M(t)=& \sum_{i=0}^p (-1)^i H_{F_i}(t) = \sum_{i=0}^p (-1)^i\sum_{j\in\mathbb{Z}}b_{ij}H_S(t-j)\\
=& \sum_{i=0}^p \sum_{j\in\mathbb{Z}} (-1)^i b_{ij}\left\{ \binom{t+n-1}{n-1}-j{\cdot}\binom{t+n-2}{n-2}\right\} \\
 &+ (\text{a polynomial in $t$ of  degree $<n-2$})
\end{align*}
and $\dim M=n$, we have $-\mathrm{e}_1 (M)=\sum_{i=0}^{p}\sum_{j\in \mathbb{Z}} (-1)^{i}{\cdot}b_{ij}{\cdot}(-j).$
\end{proof}

Let us consider the graded Bourbaki ideals.
When $\projdim_S M=1$, the ideal $I$ in \eqref{a3.1.1} can be described
as an ideal of maximal minors of a certain matrix over $S$.

Assume $M$ has rank $r>1$ and let $k\geq  t_0(M)$.
Let
\[
0\to F_1 \xrightarrow{\psi_1} F_0 \xrightarrow{\pi} M \to 0 \quad \text{ and } \quad
0 \to S(-k)^{r-1} \xrightarrow{\varphi} M \xrightarrow{\varepsilon} I(m) \to 0
\]
be a graded free resolution of $M$ and a graded Bourbaki sequence with $\height (I)\ge 2$, respectively.
Since $S(-k)^{r-1}$ is a projective module,   there exists a graded $S$-module homomorphism
 $\psi_2: S(-k)^{r-1}\to F_0$ such that $\pi\circ \psi_2=\varphi$.
Then $I$ can be obtained  as follows.

\begin{Proposition}\label{a3.4}
Assume $M$ is a finitely generated graded module which is torsionfree of rank  $r>1$ with  $\projdim_S M=1$.   With the above notation we set  $F_0=\bigoplus_{i=1}^\alpha S(-a_i)$ and $F_1=\bigoplus_{j=1}^\beta S(-b_j)$.
Then
\[
0\to F_1\oplus S(-k)^{\alpha-\beta-1} \xrightarrow{(\psi_1, \psi_2)} F_0 \xrightarrow{\varepsilon \circ \pi} I(m) \to 0
\]
 is a graded free resolution of $I(m)$, $I=I_{\alpha-1} (\psi_1, \psi_2)$, and
$$m=\sum_{j=1}^\beta b_j - \sum_{i=1}^\alpha a_i +k(\alpha-\beta-1).$$

Furthermore, if $M$ is generated in degree $k$, then there exist a matrix $A$ representing $\psi$
and a $(\beta+1) \times \beta$ submatrix $A'$ of  $A$ such that $I=I_{\beta} (A')$.
\end{Proposition}


\begin{proof}
Note that $r=\alpha-\beta$. It is easy to check that
$$F_1\oplus S(-k)^{\alpha-\beta-1} \xrightarrow{(\psi_1, \psi_2)} F_0 \xrightarrow{\varepsilon \circ \pi} I(m) \to 0$$
 is a graded exact sequence of $S$-modules. Since $\Ker (\psi_1, \psi_2)$ is a torsionfree $S$-module of rank $\beta+(\alpha-\beta-1)-\alpha+1=0$, it follows that  $\Ker (\psi_1, \psi_2)=0$.
Therefore, since $\grade (I) \geq 2$,  the Hilbert-Burch theorem \cite[Theorem 1.4.17]{BH} implies  that  $I=I_{\alpha-1} (\psi_1, \psi_2)$.
The displayed formula for $m$ follows by a simple computation from Theorem~\ref{a2.1} and Proposition~\ref{a3.3}.

Suppose that $M$ is generated in degree $k$.
Since $\psi_2: S(-k)^{\alpha-\beta-1} \to F_0=S(-k)^{\alpha}$ is an injective graded map of degree $0$, we may pick bases of
$S(-k)^{\alpha-\beta-1}$ and of $F_0$ such that the  matrix representing $\psi_2$ is
\begin{equation*}
\left(
\begin{array}{ccc}
1&& \\
&\ddots&\\
&&1\\ \hline
&&\\
&\mbox{\Large{0}}&
\end{array}
\right).
\end{equation*}
We pick any free basis of $F_1$ and we denote $A=(a_{ij})$ the matrix representing $\psi_1$ in the fixed bases.
Then
\begin{align*}
I=&I_{\alpha-1} (\psi_1, \psi_2)=I_{\alpha-1}\left(
\begin{array}{ccc|ccc}
a_{11}&\cdots&a_{1\beta}&1&& \\
\vdots&\ddots&\vdots&&\ddots&\\
a_{\alpha-\beta-1\ 1}&\cdots&a_{\alpha-\beta-1\ \beta}&&&1\\ \hline
a_{\alpha-\beta\ 1}&\cdots&a_{\alpha-\beta\ \beta}&&&\\
\vdots&\ddots&\vdots&&\mbox{\Large{0}}&\\
a_{\alpha\ 1}&\cdots&a_{\alpha\ \beta}&&&
\end{array}
\right)\\
=&I_{\beta}\left(
\begin{array}{ccc}
a_{\alpha-\beta\ 1}&\cdots&a_{\alpha-\beta\ \beta}\\
\vdots&\ddots&\vdots\\
a_{\alpha\ 1}&\cdots&a_{\alpha\ \beta}
\end{array}
\right).
\end{align*}
\end{proof}







\section{Characterization of  Bourbaki sequences}\label{section3.5}

In this section we consider maps $\varphi:R^s\to M$ and enquire whether $\Coker(\varphi)$ is a torsionfree module.
This applies, when $R$ is a normal domain and $M$ a torsionfree $R$-module of rank $r$, to characterize
the maps $\varphi:R^{r-1}\to M$ which are part of a Bourbaki sequence like \eqref{bbb}.

In the following lemma we present characterizations for torsionfree and for reflexive modules, respectively.
For the convenience of the reader, we include the proof.

\begin{Lemma}\label{a3.5}
Let $R$ be a Noetherian ring and $M$ a finitely generated $R$-module. Suppose that $R$ is generically Gorenstein, that is, $R_\pp$  is Gorenstein for all $\pp\in \Ass(R)$.
\begin{enumerate}[{\rm (a)}]
\item The following conditions are equivalent:
\begin{enumerate}[{\rm (i)}]
\item $M$ is torsionfree;
\item $\Ass(M)\subseteq \Ass(R)$;
\item there is an exact sequence $0\to M \to F$, where $F$ is a finitely generated free $R$-module;
\item  the canonical map  $\varepsilon:M\to M^{**}$  is injective.
\end{enumerate}
\item The following conditions are equivalent:
\begin{enumerate}[{\rm (i)}]
\item $M$ is reflexive;
\item  there is an exact sequence $0\to M \to F \to G$, where $F$ and $G$ are finitely generated free $R$-modules.
\end{enumerate}
\end{enumerate}
\end{Lemma}

\begin{proof}   Note that $\Ass (R)=\Min (R)$ since $R_\frk{p}$ is an Artinian  Gorenstein ring for all $\frk{p}\in\Ass(R)$.

(a) The implications
(iv)\implies (iii) $\Rightarrow$ (ii) $\Rightarrow$ (i) are clear.
Now suppose (i) holds, that is, the sequence $0\to M \to \mathrm{Q}(R)\otimes M$ is exact. Then $$\Ass (M)\subseteq \Ass (\mathrm{Q}(R)\otimes M)=\Ass(M)\cap \Ass(R)$$ since $\Ass (R)=\Min (R)$. It follows that $\Ass(M)\subseteq \Ass(R)$.

Let $\frk{p}\in \Ass(R)$. Since $R_\frk{p}$ is an Artinian  Gorenstein ring, by Matlis duality  we get that
$\varepsilon_{\frk{p}}$ is a bijective map. Thus $\Ass (R) \cap \Ass(\Ker (\varepsilon)) =\emptyset$.
Since $\Ass(\Ker(\varepsilon))\subseteq \Ass(M)\subseteq \Ass (R)$, we obtain that  the map $\varepsilon$ is injective, which proves (iv).

(b) (i) \implies (ii) is clear.

(ii) \implies (i): Suppose that $0\to M \xrightarrow{\psi} F \to G$ is an exact sequence, where $F$ and $G$ are finitely generated free $R$-modules. Set $X=\Coker(\psi)$. $X$ is torsionfree by (a).
By applying $\Hom_R(-, R)$ to the exact sequence $0\to M\xrightarrow{\psi} F\to X\to 0$, we obtain the exact sequence
\[
0\to X^{*} \to F^* \xrightarrow{\psi^*} M^* \to \Ext_R^1 (X, R) \to 0
\]
of $R$-modules.
Let $N=\Ext_R^1 (X, R)^*$. For any $\fkp \in \Ass (R)$, $R_\fkp$ is an Artinian Gorenstein ring, hence $\Ext^1_{R_\fkp}(X_\fkp, R_\fkp)=0$. The latter implies $N_\fkp = 0.$ Since $N$ is torsionfree, arguing as before, we obtain that $N=0$.

Hence, by applying $\Hom_R(-, R)$ again, we obtain the commutative diagram
\[
\xymatrix{
0 \ar[r] & M\ar[r]^{\psi} \ar[d]_{\varepsilon} &F \ar[r] \ar@{=}[d] & X\ar[r] \ar[d] & 0\\
0 \ar[r] & M^{**}\ar[r]^{\psi^{**}}  & F^{**} \ar[r] & \Coker(\psi^{**}) \ar[r] & 0\\
}
\]
with exact rows. The cokernel of the canonical map $\varepsilon: M\to M^{**}$ is isomorphic to the kernel of the map $X\to \Coker(\psi^{**})$, whence $\Coker(\varepsilon)$ is torsionfree.
Since $\varepsilon_{\frk{p}}$ is bijective for all $\frk{p}\in \Ass (R)$, $\varepsilon$ is bijective. It follows that $M$ is a reflexive module.
\end{proof}

\begin{Theorem}\label{a4.2}
Let $R$ be a normal domain of dimension $\geq 2$ and $M$ a finitely generated reflexive $R$-module, and let  $\varphi: R^{s} \to M$ be an $R$-module homomorphism.
 Let $0\to M \xrightarrow{\iota} F \to X \to 0$ be an exact sequence  of $R$-modules with  $F$  free and $X$  torsionfree.
Then the following conditions are equivalent:
\begin{enumerate}
\item[{\rm (i)}] the map $\varphi$ is injective and  $\Coker(\varphi)$ is torsionfree;
\item[{\rm (ii)}] $\height (I_{s}(\iota\circ \varphi)) \geq 2$.
\end{enumerate}
\end{Theorem}

\begin{proof} (i)\iff (ii):
We can rephrase the condition that $\varphi$ is injective as follows.
\begin{alignat*}{2}
\text{$\varphi$ is injective} \quad&\Leftrightarrow \text{$\iota\circ \varphi$ is injective} &\quad &\Leftrightarrow \text{$(\iota\circ \varphi)_{(0)}$ is injective} \\
&\Leftrightarrow \text{$I_{s}(\iota\circ \varphi)_{(0)}=R_{(0)}$} &&\Leftrightarrow \text{$\height(I_{s}(\iota\circ \varphi))>0$.}
\end{alignat*}
Hence we may assume that $\varphi$ is injective.

Set $N=\Coker (\varphi)$ and $C=\Coker (\iota\circ \varphi)$. Consider the commutative diagram
\begin{eqnarray}
\label{diagram}
\begin{split}
\xymatrix{
&&0 \ar[d]&0 \ar[d]&\\
0 \ar[r] & R^{s}\ar[r]^{\varphi} \ar@{=}[d] &M \ar[r] \ar[d]^{\iota} & N\ar[r] \ar[d] & 0\\
0 \ar[r] & R^{s}\ar[r]_{\iota\circ \varphi}  & F \ar[r] \ar[d] & C \ar[r] \ar[d] & 0 \\
&&X\ar@{=}[r] \ar[d]&X \ar[d]&\\
&&0&0&.\\
}
\end{split}
\end{eqnarray}
Then
\begin{center}
$N$ is torsionfree $\Leftrightarrow$ $\Ass (N) \subseteq\{0\}$ $\Leftrightarrow$ $\Ass (C) \subseteq\{0\}$ $\Leftrightarrow$ $C$ is torsionfree,
\end{center}
where the first and third equivalence follow from Lemma \ref{a3.5} and the second equivalence follows from the inclusions
$$\Ass (N) \subseteq \Ass (C)\subseteq \Ass (N) \cup \Ass (X) \subseteq \Ass (N) \cup \{0\}.$$ 

On the other hand, by \cite[Proposition 1.4.1(a)]{BH},  $C$ is torsionfree if and only if $\depth C_\frk{p}>0$ for all $\frk{p}\in\Spec R$ with $\height (\frk{p})\ge 1$.
Let $\frk{p}$ be a prime ideal of $R$. By (\ref{diagram}), $\projdim _{R_\frk{p}} C_\frk{p} \le 1$. From the Auslander-Buchsbaum theorem we obtain $\depth_{R_\frk{p}} C_\frk{p} = \depth R_\frk{p} - \projdim_{R_{\frk{p}}} C_\frk{p}$.

If $\height(\frk{p}) \geq 2$, since $R$ is a normal ring it satisfies Serre's condition $(S_2)$, hence $\depth {R_\frk{p}} \geq 2$.
Thus  $\depth_{R_\frk{p}} C_\frk{p} >0$.

If $\height(\frk{p})=1$, since $R$ satisfies the condition $(R_1)$ we get that $R_{\frk{p}}$ is a regular local ring. Consequently,
$\depth_{R_\frk{p}} C_\frk{p} = 1 - \projdim_{R_{\frk{p}}} C_\frk{p}$.

Hence, the torsionfreeness of $C$ is equivalent to saying that $C_\frk{p}$ is $R_\frk{p}$-free for any height one prime ideal $\frk{p}$ of $R$. By \cite[Proposition 1.4.10]{BH} the latter condition is equivalent to saying that $(\iota\circ \varphi)_\frk{p}$ is a split monomorphism for any height one prime $\frk{p}$, which in turn is equivalent to  $I_{s}(\iota\circ \varphi)_\frk{p}=R_\frk{p}$ for any height one prime $\frk{p}$.

Therefore, $N$ is a torsionfree module if and only if $\height (I_{s}(\iota\circ \varphi))\geq 2$.
\end{proof}

Assume $R$ is graded and $M$ is a graded $R$-module of rank  $r>1$. We will prove that under mild assumptions on $R$ and $M$, given $k \gg 0$ and a $K$-basis $m_1,\dots, m_\alpha$ of $M_k$, then any $r-1$ generic $K$-linear combinations of $m_1,\dots, m_\alpha$ generate a free submodule of $M$ which is the beginning of a graded Bourbaki sequence of $M$.

\begin{Lemma}\label{b4.3}
Let $R=\bigoplus_{n\geq 0} R_n$ be a graded Noetherian ring such that $R_0=K$ is an algebraically closed field.
Let $T=R[z_1, \dots, z_m]$ be the polynomial ring over $R$. We regard $T$ as a graded ring by using the grading of $R$ and $\deg(z_i)=0$ for $1\le i\le m$. Let $I$ be a graded ideal of $T$ such that $I\subseteq \bigoplus_{n> 0} T_n$ and set $S=T/I$. For $\lambda=(\lambda_1, \dots, \lambda_m)\in K^m$, we denote by $\fkp_\lambda$ the maximal ideal $(z_1-\lambda_1, \dots, z_m-\lambda_m)$  of $K[z_1, \dots, z_m]$. Then
\[
\{\lambda\in K^m \mid \dim S/\fkp_\lambda S<e\}
\]
is a Zariski open subset of $K^m$ for any integer $e$.
\end{Lemma}

\begin{proof}
Set $R'=S_0=T_0=K[z_1, \dots, z_m]$.
By semicontinuity of fiber dimension (see for example \cite[Theorem 14.8, b]{Ei}), for any integer $e$, there exists an ideal $J_e$ of $R'$ such that
\[
\{ \fkp\in \Spec R' \mid \dim \mathrm{Q}(R'/\fkp)\otimes_{R'} S  \ge e\}=V(J_e).
\]
For $\lambda\in K^m$, since $\mathrm{Q}(R'/\fkp_{\lambda})\otimes_{R'} S\cong R'/\fkp_{\lambda}\otimes_{R'}S \cong S/\fkp_\lambda S$, it follows that $\dim S/\fkp_\lambda S\ge e$ if and only if $J_e\subseteq \fkp_\lambda$. Hence, since $K$ is algebraically closed,
\[
\{\lambda\in K^m \mid \dim S/\fkp_\lambda S<e\}
\]
is a Zariski open subset of $K^m$.
\end{proof}



Combining Theorem \ref{a4.2} and Lemma \ref{b4.3} yields the following result.

\begin{Theorem}\label{b4.5}
Let $R=\bigoplus_{n\geq 0} R_n$ be a  standard graded Cohen-Macaulay normal domain  of dimension $d\ge 2$ such that $R_0=K$ is an algebraically closed field.
Let $M=\bigoplus_{n\in \mathbb{Z}} M_n$ be a finitely generated reflexive graded $R$-module of rank  $r>1$.

Let $k\geq t_0(M)$,  $F=R(-k)^{r-1}$,  $G=R(-k)^{\alpha}\xrightarrow{\pi} M_{\ge k}$ a graded  surjective map, and $\iota: M_{\ge k}\to M$ the inclusion map. Fix free bases $f_1,\dots, f_{r-1}$ and $g_1, \dots, g_{\alpha}$ of $F$ and $G$, respectively.
For $\lambda=(\lambda_{ij})\in K^{\alpha \times (r-1)}$, let $\varphi_{\lambda}$ be the graded $R$-module homomorphism $F\to G$ such that $\varphi_{\lambda}(f_j)=\sum_{i=1}^{\alpha}\lambda_{ij}g_i$ for $1\le j\le r-1$.

With these assumptions and notation,  the set
{\small
\[
\{ \lambda\in K^{\alpha \times (r-1)} \mid 0\to F\xrightarrow{\iota\circ\pi\circ \varphi_{\lambda}}M\to \Coker(\iota\circ\pi\circ \varphi_{\lambda}) \to 0 \text{ is a Bourbaki sequence of $M$}\}
\]
}
is a nonempty Zariski open subset of $K^{\alpha \times (r-1)}$.
\end{Theorem}

\begin{proof}
Our assumptions on $M$ imply that for any $\ell\ge t_0(M^*)$, there exists a graded exact sequence $0\to~M\xrightarrow{\psi} H \to X \to 0$ of $R$-modules  such that $H$ is a free $R$-module generated in single degree $-\ell$ of rank $\beta$ and $X$ is torsionfree. 
 Indeed, by dualizing the short exact sequence  
\[
0 \to (M^*)_{\ge \ell} \to M^* \to M^*/(M^*)_{\ge \ell} \to 0
\]
we get the exact sequence
\[
0 \to \Hom_R (M^*/(M^*)_{\ge \ell}, R) \to M^{**} \to ((M^*)_{\ge \ell})^* \to \Ext_R^1 (M^*/(M^*)_{\ge \ell}, R).
\]
Since $R$ is a Cohen-Macaulay ring of dimension $\ge 2$ with unique graded maximal ideal, we have $\Hom_R (M^*/(M^*)_{\ge \ell}, R)=0$ and $\Ext_R^1 (M^*/(M^*)_{\ge \ell}, R)=0$, thus 
\begin{align}\label{final}
M^{**} \cong ((M^*)_{\ge \ell})^*.  
\end{align}
Let $H_1 \to H_0 \to (M^*)_{\ge \ell} \to 0$ be a minimal graded free presentation of $(M^*)_{\ge \ell}$. Applying $\Hom_R(-,R)$ to it and using \eqref{final} we obtain the exact sequence  $0 \to M \to H_0^* \to H_1^*$. This is the exact sequence that we want since $H_0$ is a free $R$-module generated in single degree $\ell$.
Let $h_1, \dots, h_\beta$ be a basis of $H$.

Let $T=R[z_{ij}]_{1\le i\le \alpha, 1\le j \le r-1}$ be the graded polynomial ring over $R$ with $\deg(z_{ij})=0$, and let $A$ denote the matrix representing $\psi\circ \iota\circ\pi :G=R(-k)^\alpha \to H=R(\ell)^\beta$ with respect to the bases $g_1, \dots,g_\alpha$ and $h_1, \dots, h_\beta$. Note that $I=I_{r-1}(A{\cdot}(z_{ij}))$ is a graded ideal of $T$ generated in degree $(\ell+k)(r-1)$. Hence after choosing $\ell$ large enough, $(\ell+k)(r-1)$ is positive, thus $I\subseteq  \bigoplus_{n> 0} T_n$.

Then, by Lemma \ref{b4.3},
\begin{align*}
\{\lambda\in K^{\alpha \times (r-1)}  \mid \dim T/(I+\fkp_\lambda T)< \dim R-1\}
\end{align*}
is a Zariski open subset of $K^{\alpha \times (r-1)}$. Since $T/(I+\fkp_\lambda T)\iso R/I_{r-1}(\psi\circ \iota\circ\pi\circ \varphi_{\lambda})$, 
by Theorem \ref{a4.2}, it follows that
{\small
\[
\{ \lambda\in K^{\alpha \times (r-1)} \mid 0\to F\xrightarrow{\iota\circ\pi\circ \varphi_{\lambda}}M\to \Coker(\iota\circ\pi\circ \varphi_{\lambda}) \to 0 \text{ is a Bourbaki sequence of $M$}\}
\]
}
is a Zariski open subset of $K^{\alpha \times (r-1)}$ and nonempty by Theorem \ref{basic2.1}.
\end{proof}


\begin{Theorem}\label{a4.3}
Let $R$ be a Noetherian normal domain with $\dim R\geq 2$ and $N$ an $R$-module of rank $s$ with finite free presentation $R^{\beta_1} \xrightarrow{\psi} R^{\beta_0} \to N \to 0$.
\begin{enumerate}[{\rm (a)}]
\item If $N$ is torsionfree, then  $\height(I_{\beta_0-s} (\psi))\geq 2$.
\item Suppose that  there exists an exact sequence
\[
0\to R^{s} \xrightarrow{\varphi} M \to N \to 0,
\]
where $M$ is a torsionfree module of finite projective dimension. If $\height(I_{\beta_0-s} (\psi))\geq 2$, then $N$ is torsionfree.
\end{enumerate}
\end{Theorem}

\begin{proof}
(a)
Let $\frk{p}\in \Spec R$ with $\height(\frk{p})= 1$.
 By \cite[Proposition 1.4.1(a)]{BH} and Lemma~\ref{a3.5},
the torsionfreeness of $N$ implies that $\depth_{R_\frk{p}} N_\frk{p}>0$.
Since the ring $R$ is normal,  it follows that $R_\fkp$ is a regular local ring of dimension $1$.
By the Auslander-Buchsbaum theorem, $ \projdim_{R_\fkp} N_\fkp = \depth R_\fkp- \depth N_\fkp$, hence $N_\fkp$ is a free $R_\fkp$-module. It follows from \cite[Proposition 1.4.9]{BH} that  $I_{\beta_0-s} (\psi)_\fkp=R_{\frk{p}}$. We conclude that  $\height(I_{\beta_0-s} (\psi))\geq 2$.

(b)
Let $\fkp \in \Spec(R)$ with $\height (\fkp) \geq 2$. Since $R$ satisfies the condition $(S_2)$ we have that $\depth R_\fkp \geq 2$.
If $M_\fkp$ is a free $R_\fkp$-module, then $\projdim_{R_\fkp} N_\fkp \leq 1$ and hence $\depth_{R_\fkp} N_\fkp=\depth R_\fkp-\projdim_{R_\fkp} N_\fkp \geq 1$.

If $M_\fkp$ is not free, then considering the mapping cone of $\varphi_\fkp$ we obtain that
$$\projdim_{R_\fkp} N_\fkp = \projdim_{R_\fkp} M_\fkp <\infty.$$
 Using the Auslander-Buchsbaum formula again, we get that $\depth_{R_\fkp} M_\fkp=\depth_{R_\fkp} N_\fkp$. On the other hand, since $M$ is torsionfree, one has $\depth_{R_\fkp} M_\fkp \geq 1$. Therefore, we have the following chain of equivalences:
\begin{align*}
\text{$N$ is torsionfree} &\Leftrightarrow \text{$\depth_{R_\frk{q}} N_\frk{q}>0$ for all $\frk{q}\in \Spec R$ with $\height (\frk{q})\ge 1$}\\
&\Leftrightarrow \text{$\depth_{R_\frk{q}} N_\frk{q}>0$ for all $\frk{q}\in \Spec R$ with $\height (\frk{q})= 1$} \\
&\Leftrightarrow \text{$N_\frk{q}$ is $R_\frk{q}$-free for all $\frk{q}\in \Spec R$ with $\height (\frk{q})= 1$} \\
&\Leftrightarrow \text{$I_{\beta_0-s}(\psi)_\frk{q}=R_\frk{q}$ for all $\frk{q}\in \Spec R$ with $\height (\frk{q})= 1$} \\
&\Leftrightarrow \text{$\height (I_{\beta_0-s} (\psi))\ge 2$,}
\end{align*}
where   the second equivalence follows from the above argument. For the third equivalence we use that when $\height(\fkq)=1$ the ring $R_\fkq$ is regular, so $\projdim_{R_\fkq} N_\fkq =1-\depth_{R_\fkq} N_\fkq \in \{0,1\}$.
\end{proof}

The next corollary is an immediate consequence of the previous theorem.

\begin{Corollary}
\label{nice}
Let $R$ be a Noetherian normal domain with $\dim R \geq 2$ and  $M$ a finitely generated torsionfree $R$-module of rank $r>1$ with $\projdim M <\infty$. Let $\varphi:R^{r-1}\to M$ be an injective $R$-module homomorphism, and $R^{\beta_1} \xrightarrow{\psi} R^{\beta_0} \to \Coker (\varphi) \to 0$  a presentation for $\Coker (\varphi)$.
Then the module $\Coker (\varphi)$ is torsionfree if and only if $\height(I_{\beta_0-r+1}(\psi)) \geq 2$.
\end{Corollary}

The next theorem tells us how to compute a Bourbaki ideal, once we have its relation  matrix.

\begin{Theorem}\label{a4.4}
Let $R$ be a Noetherian factorial domain and $I$ an ideal of $R$ of grade $\geq 2$.
Suppose $R^{\beta} \xrightarrow{\psi} R^{\alpha} \xrightarrow{\varepsilon} I \to 0$ is a finite free presentation  of $I$,
and let $B$ be a matrix  representing $\psi$. Let $C$ be  any $\alpha\times (\alpha-1)$ submatrix of $B$ of maximal rank $\alpha-1$. Then there exists  a unique element $x\in R$ such that  $I=(1/x)I_{\alpha-1}(C)$.
\end{Theorem}

\begin{proof}
Taking the $R$-dual of  $R^{\beta} \xrightarrow{\psi} R^{\alpha} \xrightarrow{\varepsilon} I \to 0$ yields the exact sequence
\[
0\rightarrow R\xrightarrow{\gamma} R^\alpha  \xrightarrow{\psi^*} R^\beta.
\]

With respect to the canonical bases, $B^{\mathrm{T}}$ represents $\psi^*$, and let the vector $u=(u_1,\dots, u_\alpha)^{\mathrm{T}}$ represent the map $\gamma$.
Then ${B}^{\mathrm{T}}u=0$,  and  the elements $u_1,\ldots,u_\alpha$ generate the ideal $I$.
Since  ${C}^{\mathrm{T}}$ is a submatrix of size $(\alpha-1)\times \alpha$ of  ${B}^{\mathrm{T}}$, which is of size $\beta\times \alpha$, it follows that also ${C}^{\mathrm{T}}u=0$.
 
Let $\Delta_i({C}^{\mathrm{T}})$ be the determinant of the matrix which is obtained by deleting the  $i$th column of ${C}^{\mathrm{T}}$, and set $f_i=(-1)^{i+1}\Delta_i({C}^{\mathrm{T}})$ for $i=1,\dots, \alpha$. Then
\[
{C}^{\mathrm{T}}{\cdot}\left(\begin{smallmatrix}
f_1\\
f_2\\
\vdots \\
f_{\alpha}
\end{smallmatrix}
\right)=0.
\]
Since $\rank (\Ker ({C}^{\mathrm{T}}))= \alpha-\rank ({C}^{\mathrm{T}})=1$, it follows that
\[
I=c \cdot (f_1,f_2,\ldots,f_{\alpha})
\]
for a unique  nonzero element $c\in \mathrm{Q}(R)$. As $R$ is a factorial ring,  the greatest common divisor of $f_1,\ldots,f_{\alpha}$ exists.  Let $g=\gcd\{f_1,\ldots,f_{\alpha}\}$. Then $I=cgJ$ where $J=(1/g)(f_1,f_2,\ldots,f_{\alpha})$ is an ideal in $R$  of grade  $\geq 2$.  Since both ideals $I$ and $J$  have grade  $\geq 2$, it follows that $cg$ is a unit element in $R$. The desired conclusion follows with $x=1/c$.
\end{proof}


\section{Bourbaki ideals of Koszul cycles}\label{section4}

In this section we study  Bourbaki ideals of Koszul cycles, and compute them explicitly in some cases.
Let   $S=K[x_1, x_2, \dots, x_n]$ be a polynomial ring of dimension $n\ge 2$ over a field $K$, which is not necessarily infinite.
 Let $\mm=(x_1, x_2, \dots, x_n)$  be the maximal  graded  ideal of $S$ and
\[
0\to K_n \xrightarrow{\partial_n} K_{n-1} \xrightarrow{\partial_{n-1}} \dots \xrightarrow{\partial_2} K_1 \xrightarrow{\partial_1} K_0 \xrightarrow{\partial_0} S/\frk{m} \to 0
\]
be the Koszul complex on the sequence $x_1, x_2, \dots, x_n$, which is a linear free graded resolution  of $S/\mathfrak{m}\cong K$.

We set $Z_i=\Im \partial_i$ for $1\le i \le n$, which are torsionfree graded $S$-modules generated in degree $i$.
Clearly, $Z_1\cong \frk{m}$ and $Z_{n}\cong R(-n)$. Hence we focus on the Bourbaki ideals of $Z_i$ where $2\le i \le n-1$.

For $2\le i \le n-1$, we denote $r_i=\rank Z_i$.
By Theorem~\ref{basic2.1} and Lemma~\ref{b2.2}, there exist a graded ideal $I$  and a graded exact sequence
\begin{eqnarray}
\label{Bourbaki sequence of Koszul}
0\to S^{r_i-1}(-i) \to Z_i \to I(m_i) \to 0
\end{eqnarray}
such that $\height(I)= 2$, where the integer  $m_i$ is computed in Theorem~\ref{a2.1}.

Furthermore,  Proposition \ref{b2.3} implies that $S/I$ is Cohen-Macaulay on the punctured spectrum of $S$.

Let $\eb_1,\dots, \eb_n$ be a basis of  the free $S$-module $K_1$.  For any subset $\{i_1, i_2, \dots, i_k\}$ of $\{1, 2, \dots, n\}$,
we denote the wedge product $\mathbf{e}_{i_1}\wedge \mathbf{e}_{i_2} \wedge \cdots \wedge \mathbf{e}_{i_k}$ by $\mathbf{e}_{i_1 i_2 \dots i_k}$.
For any $k\geq 1$, the elements $\eb_{i_1 i_2 \dots  i_k}$ with $1\leq i_1<\dots <i_k\leq n$ form a basis for $K_k$.
For convenience, we denote by $\widehat{\mathbf{e}}_{i_1 i_2 \dots i_k}$ the unique canonical basis element in $K_{n-k}$ such that $\mathbf{e}_{i_1 i_2 \dots i_k} \widehat{\mathbf{e}}_{i_1 i_2 \dots i_k}=\mathbf{e}_{12\dots n}$.

Once we assign to the variable $x_i$ the multidegree $$\mdeg(x_i)=\mathbf{e}'_i= (0,\dots,0,1,0,\dots, 0) \in \ZZ^n,$$ for $i=1,\dots, n$, and we let $\mdeg(\mathbf{e}_{i_1,\dots, i_j})=  \mathbf{e}'_{i_1}+\dots+\mathbf{e}'_{i_j}$ when $1\leq i_1<\dots <i_j\leq n$, it is clear that the differentials in the Koszul complex described above respect this multigrading.


\begin{Lemma}\label{a5.1}
Let $2\le i\le n-1$. For a graded Bourbaki sequence {\rm (\ref{Bourbaki sequence of Koszul})}, we have
\begin{enumerate}[{\rm (a)}]
\item $r_i=\rank (Z_i)=\binom{n-1}{i-1}$ and
\item 
$m_i= i \binom{n-1}{i-1}-n\binom{n-2}{i-2}-i$.
\end{enumerate}

Hence, $I$ is generated in degree 
$i \binom{n-1}{i-1}-n\binom{n-2}{i-2}.$
\end{Lemma}

\begin{proof}
(a) Considering the free resolution of $Z_i$ which results from the Koszul resolution of $K$ we obtain that $\rank (Z_i)=\sum_{k=0}^{i-1}(-1)^k\binom{n}{i-1-k}$. The desired result follows then by induction on $i$.

(b) By Theorem \ref{a2.1} and Proposition \ref{a3.3}, we get
\[
m_i=i\left\{\binom{n-1}{i-1}-1\right\} - \mathrm{e}_1 (Z_i)\quad \text{and} \quad \mathrm{e}_1 (Z_i)=\sum_{j=i}^n(-1)^{j-i}j\binom{n}{j}.
\]
Since $j\binom{n}{j}=n\binom{n-1}{j-1}$, $\mathrm{e}_1 (Z_i)=n\sum_{j=i}^n(-1)^{j-i}\binom{n-1}{j-1}=n\binom{n-2}{i-2}$. It follows that $m_i= i \binom{n-1}{i-1}-n\binom{n-2}{i-2}-i$.
\end{proof}

\subsection{Bourbaki ideals of $Z_{n-1}$ and $Z_{n-2}$}

In this subsection we investigate the Bourbaki ideals of $Z_{n-1}$ and $Z_{n-2}$. In this case, we show that we can choose a Bourbaki sequence
\[
0 \to F \xrightarrow{\partial_i |_F} Z_{i} \to Z_{i}/\partial_i(F) \to 0
\]
so that $F$ is a submodule of $K_{i}$ generated by a part of the canonical basis of $K_{i}$. It will follow that the Bourbaki ideal $I\iso Z_{i}/\partial_i(F)$ is a monomial ideal.

We first treat the case $i=n-1$. By Lemma~\ref{a5.1}, $r_{n-1}=\rank(Z_{n-1})= n-2$.  Note that $\projdim Z_{n-1}=1$, hence we may apply Proposition \ref{a3.4}.

\begin{Proposition}\label{a5.2}
For all $1\le i<j\le n$,  $(x_i, x_j)$ is a Bourbaki ideal  of $Z_{n-1}$.
\end{Proposition}

\begin{proof}
Let $F$ be the submodule of $K_{n-1}$ generated  by the elements $\widehat{\mathbf{e}}_{k}$ with  $k\in \{1, 2, \dots, n\}\setminus \{i, j\}$.
We denote $\varphi: F\to Z_{n-1}$ the restriction of  $\partial_{n-1}$ to $F$.

We claim that $\varphi(F)\iso  S(-n+1)^{n-2}$.
Indeed, suppose that
\[
\sum_{k\in \{1, 2, \dots, n\}\setminus \{i, j\}} f_k{\cdot}\partial(\widehat{\mathbf{e}}_{k})=0,
\]
where $f_k\in S$. On the left hand side of the above equation, the coefficient of $\widehat{\mathbf{e}}_{ik}$ is $f_k x_i$ up to sign, hence $f_k=0$ for all $k$. Then from the chain of isomorphisms
\[
\Coker(\varphi)\cong K_{n-1}/(\Im(\partial_{n}) + F)\cong S\widehat{\mathbf{e}}_{i} \oplus S\widehat{\mathbf{e}}_{j}/\left< x_{i}\widehat{\mathbf{e}}_{i} \pm x_{j}\widehat{\mathbf{e}}_{j} \right>\cong (x_i, x_j)(-n+2)
\]
we derive the graded Bourbaki sequence
\[
0\to F \xrightarrow{\varphi} Z_{n-1} \to (x_i, x_j)(-n+2) \to 0,
\]
as desired.
\end{proof}


We now present a Bourbaki sequence of $Z_{n-2}$. By Lemma~\ref{a5.1}, we have  that $r_{n-2}=\rank(Z_{n-2})= \binom{n-1}{2}$.

\begin{Proposition}\label{c5.3}
Let $n\ge 3$ and $\xb=\prod_{i=1}^nx_i$. Then the ideal
$$
I=\left(  \frac{\xb}{x_ix_j}\ \middle| \ (i, j)\in \{(1, 2), (2, 3), \dots, (n-1, n), (n, 1)\} \right)
$$
is a Bourbaki ideal of $Z_{n-2}$.
\end{Proposition}

\begin{proof}
We consider the set of ordered pairs  $A=\{(1, 2), (2, 3), \dots, (n-1, n), (1, n)\}$.
Let $F$ be the submodule of $K_{n-2}$ generated  by the elements $\widehat{\mathbf{e}}_{ij}$ with  $1\le i<j\le n$ and $(i, j)\not\in A$. Let $r=r_{n-2}$. Note that $\rank(F)=r-1$.

Let $\iota: Z_{n-2}\to K_{n-3}$ be the inclusion map and $\varphi: F\to Z_{n-2}$ the restriction of  $\partial_{n-2}$ to $F$.
Then,  by Theorem \ref{a4.2},
$\varphi$ is injective and $\Coker(\varphi)$ is a torsionfree module if and only if the $\height (I_{r-1} (\iota\circ \varphi)) \geq 2$.

By our choice of the basis of $F$,  the map $\iota \circ \varphi\:  F\to K_{n-3}$ is a multigraded $S$-module homomorphism.
So, if $D$ is the  matrix representing $\iota \circ \varphi$ with respect to the canonical bases, then $I_{r-1} (D)$ is  a monomial ideal.
To conclude that $\height (I_{r-1}(D))$ is at least two,  it suffices to show that no $x_i$ divides all the monomial generators of $I_{r-1} (D)$. By symmetry, it is enough to show this for $i=n$. In other words,  we have  to show that there exists an $(r-1)\times (r-1)$ submatrix of $D$ such that its determinant is not divisible  by $x_n$.

Let $F_1$ and $F_2$ be the  submodules of $F$ generated by the elements
\begin{center}
$\widehat{\mathbf{e}}_{ij}$ with  $(i,j)\in \{(1,3), (1,4), \dots, (1, n-1)\}$,  and \\
$\widehat{\mathbf{e}}_{ij}$ with  $2\le i<j\le n$ and $|j-i|\ne1$, respectively.
\end{center}
Then $F=F_1\oplus F_2$.
Let $G_1$ and $G_2$ be the  free submodules of $K_{n-3}$ generated by the elements
\begin{center}
$\widehat{\mathbf{e}}_{1ij}$ with  $(i,j)\in \{(2,3), (3,4), \dots, (n-2, n-1)\}$,  and \\
$\widehat{\mathbf{e}}_{1ij}$ with  $2\le i<j\le n$ and $|j-i|\ne1$, respectively.
\end{center}
Note that  $\rank (G_1\oplus G_2)=(n-3) + \binom{n-2}{2}=r-1$. With respect to the above specified bases, the matrix representing the composition $F=F_1\oplus F_2 \to K_{n-3}\to G_1\oplus G_2$   has the form
$\left(\begin{smallmatrix}
A_{11} & 0\\
A_{21} & A_{22}
\end{smallmatrix}\right),$ where
\[
A_{11}=\left(\begin{matrix}
x_2 &      &            &           & \\
x_4 & x_3 &            &  \mbox{\LARGE{0}} & \\
     & x_5 & \ddots  &           & \\
     &      & \ddots  & x_{n-3}  & \\
     &  \mbox{\LARGE{0}} &            & x_{n-1} & x_{n-2}
\end{matrix}\right) \text{ and }
A_{22}=\left(\begin{matrix}
x_1 &      &            &           & \\
     & x_1 &            &      \mbox{\LARGE{0}}     & \\
     &      & \ddots  &           & \\
     &      &            & x_{1}  & \\
     &  \mbox{\LARGE{0}}  &            &          & x_{1}
\end{matrix}\right).
\]
Hence the determinant of $\left(\begin{smallmatrix}
A_{11} & 0\\
A_{21} & A_{22}
\end{smallmatrix}\right)$ is $x_1^{\binom{n-1}{2}}x_2x_3\cdots x_{n-2}$, which is not divisible by $x_n$.
Therefore,
the sequence $0\to F \xrightarrow{\varphi} Z_{n-2} \to \Coker(\varphi)\to 0$ is exact and $\Coker(\varphi)$ is a torsionfree module of rank one.

Set $N=\Coker(\varphi)$.
Then, since $N\iso K_{n-2}/(\Im \partial_{n-1} + F)$, for $N$ we have the graded  free presentation
\[
K_{n-1} \xrightarrow{H} \bigoplus_{(i,j)\in A}S\widehat{\mathbf{e}}_{ij} \to N \to 0,
\]
where the matrix
\[
H=\left(\begin{matrix}
x_2 & x_1 &            &           & & \\
   0  & -x_3 & x_2      &           & \mbox{\LARGE{0}}&  \\
  0  &   0   & -x_4  &   \ddots  & & \\
  \vdots   &   \vdots   &     \ddots    & \ddots  & x_{n-2} & \\
  0   &  0    &    \cdots        &   \ddots     & -x_{n} & x_{n-1}\\
-x_n&   0   &     \cdots      &    \cdots  &    0  & -x_{1}
\end{matrix}\right)
\]
gives the map  with respect to the canonical bases of $K_{n-1}$ and $\bigoplus_{(i,j)\in A}S\widehat{\mathbf{e}}_{ij}$.

Recall that  $I$  is the ideal generated by the monomials $\xb/x_ix_j$ with $(i, j)\in A$. We denote  $\psi:~\bigoplus_{(i,j)\in A}S\widehat{\mathbf{e}}_{ij} \to I$ the graded $S$-module homomorphism letting $\psi(\widehat{\mathbf{e}}_{ij})=\xb/x_ix_j$
for all $(i,j)\in A$.
Then, since $\Im (H)\subseteq \Ker(\psi)$, $\psi$ induces the surjection
\[
\overline{\psi}: N=\bigoplus_{(i,j)\in A}S\widehat{\mathbf{e}}_{ij}/\Im H \to I.
\]
Hence $N\iso I$ since the Kernel of $\overline{\psi}$ is torsionfree of rank zero.
Therefore, we have the graded Bourbaki sequence
$0\to F \xrightarrow{\varphi} Z_{n-2} \to I \to 0,$
as desired.
\end{proof}

\subsection{When does $Z_i$ have a multigraded Bourbaki sequence?}

For $i=n-1$ or $i= n-2$, in the previous subsection we found free submodules $F\subset K_i$ such that
$0\to F \xrightarrow{\partial_i|_F}  Z_i\to I\to 0$ is a graded Bourbaki sequence and $F$ is generated by a subset of the canonical basis of $K_i$.
In this situation,  $Z_i$ admits a multigraded Bourbaki sequence with respect to the multigrading enherited from the Koszul complex  described at the beginning of Section \ref{section4}.


However, the next   result shows  that   multigraded Bourbaki sequences only exist for special values of $i$ and $n$.

\begin{Theorem}\label{a5.4}
Let $2\le i \le n-1$ and $F$ a free  submodule of $K_i$  generated by a subset of the canonical basis  of $K_i$
such that  the sequence
\begin{equation}\label{Bourbaki}
0\to F \xrightarrow{\partial_i|_F} Z_i \to Z_i/\partial_i(F) \to 0
\end{equation}
is  a graded Bourbaki sequence of $Z_i$. Then
\begin{equation}
\label{i2v}
i\geq \max\left\{   i\binom{n-1}{i-1}-n\binom{n-1}{i-2}, (n-i)\binom{n-1}{i} - n\binom{n-1}{i+1} \right\}.
\end{equation}
 \end{Theorem}

\begin{proof}
To show the conclusion, we prove that $i$ is at least each of the terms on the right hand side of the inequality \eqref{i2v}.

For $1\leq j \leq n$ we set
$
\Ic_j=\{ (\ell_1,\dots, \ell_j) \mid  1\leq \ell_1<\dots <\ell_j \leq n   \}.
$
Then $\{\eb_\pb \mid \pb \in \Ic_j\}$ is the canonical basis of $K_j$.
We assume that
$$
H=\{\eb_{\pb_1},\dots, \eb_{\pb_{r-1}}\}
$$
is an $S$-basis of $F$, where $r=\rank(Z_i)=\binom{n-1}{i-1}$ and $\pb_1,\dots, \pb_{r-1} \in \Ic_i$.

 Let $\varphi: F\to K_{i-1}$ be the composition of the inclusion $F\to K_i$ and $\partial_i$. Let $A$ be a matrix representing $\varphi$ with respect to the specified bases of $F$ and $K_{i-1}$.

Then, since the basis of $F$ is a part of the canonical basis of $K_i$, $\varphi$ is a multigraded $S$-module homomorphism. This shows that each $(r-1)$-minor of $A$ is a monomial.
Let $\Delta \left[ _{\mathbf{e}_{\qb_1}, \dots, \mathbf{e}_{\qb_{r-1}}}^{\mathbf{e}_{\pb_1}, \dots, \mathbf{e}_{\pb_{r-1}}}\right]$ denote the determinant of the submatrix of $A$ with respect to the columns indexed by $\mathbf{e}_{\pb_1}, \dots, \mathbf{e}_{\pb_{r-1}}$ and the rows indexed by $\mathbf{e}_{\qb_1}, \dots, \mathbf{e}_{\qb_{r-1}}$. Then
\[
I_{r-1}(\varphi)=I_{r-1}(A)=\left(\Delta \left[ _{\mathbf{e}_{\qb_1}, \dots, \mathbf{e}_{\qb_{r-1}}}^{\mathbf{e}_{\pb_1}, \dots, \mathbf{e}_{\pb_{r-1}}}\right] \ \middle| \ \qb_1,\dots, \qb_{r-1} \in \Ic_{i-1}\right),
\]
and the multidegree (actually the exponent) of the monomial $\Delta \left[ _{\mathbf{e}_{\qb_1}, \dots, \mathbf{e}_{\qb_{r-1}}}^{\mathbf{e}_{\pb_1}, \dots, \mathbf{e}_{\pb_{r-1}}}\right]$  is $\sum_{j=1}^{r-1} \mdeg (\mathbf{e}_{\pb_j}) - \sum_{j=1}^{r-1} \mdeg (\mathbf{e}_{\qb_j})$.

Assume
$
\mdeg (  \sum_{j=1}^{r-1} \eb_{\pb_j} )=(a_1,\dots, a_n) \in \ZZ^n.
$
 Since ${\pb_j} \in \Ic_i$ for $j=1,\dots, r-1$, it follows that
$$
a_1+\dots+ a_n= i(r-1).
$$

Since (\ref{Bourbaki}) is a Bourbaki sequence of $Z_i$, Theorem \ref{a4.2} implies that $\height I_{r-1}(\varphi)\ge 2$. As
 $I_{r-1}(A)$ is a monomial ideal,
\[
\begin{split}
&\height I_{r-1}(A)\ge 2\\
\Leftrightarrow  &\gcd \left\{ \Delta \left[ _{\mathbf{e}_{\qb_1}, \dots, \mathbf{e}_{\qb_{r-1}}}^{\mathbf{e}_{\pb_1}, \dots, \mathbf{e}_{\pb_{r-1}}}\right] \ \middle| \  \qb_1,\dots, \qb_{r-1} \in \Ic_{i-1} \right\} =1\\
\Leftrightarrow & \text{ for all $1\le k \le n$, there exist $\qb_1^{(k)},\dots, \qb_{r-1}^{(k)} \in \Ic_{i-1}$ such that
$\Delta \left[ _{\mathbf{e}_{\qb_1^{(k)}}, \dots, \mathbf{e}_{\qb_{r-1}^{(k)}}}^{\mathbf{e}_{\pb_1}, \dots, \mathbf{e}_{\pb_{r-1}}}\right]$ is not} \\
&\text{divisible by $x_k$.} \\
\end{split}
\]

Therefore, for any $1\le k \le n$,  $a_k$ equals the $k$th component of $ \sum_{j=1}^{r-1} \mdeg(\eb_{\qb_j^{(k)}}  )$, whence $a_k \leq  |\{ \qb \in \Ic_{i-1} \mid k\in \qb \}| = \binom{n-1}{i-2}$. Adding these relations for all $k$ we obtain $\sum_{i=1}^n a_k \leq n \binom{n-1}{i-2}$, from which we get that $i\geq   i\binom{n-1}{i-1}-n\binom{n-1}{i-2} $.

To verify the second inequality subsumed  by \eqref{i2v} we let $\psi: K_{i+1} \to K_{i}/F$ be the composition  of $\partial_{i+1}$ and the canonical map $K_i \to K_i/F$.  Then $\psi$ is a multigraded $S$-module homomorphism. 

Set $s=\rank(K_i/F)-1= \binom{n}{i}-\left\{\binom{n-1}{i-1}-1\right\} -1=\binom{n-1}{i}$. We identify $K_i/F$ with the  free $S$-module with the basis $\eb_{\qb_1},\dots, \eb_{\qb_{s+1}}$, where
\[
\Ic_{i}\setminus H=\{\qb_1, \dots, \qb_{s+1}\}.
\]

If we let $\sum_{j=1}^{s+1} \mdeg(\eb_{\qb_j})=(c_1,\dots, c_n)$, then $\sum_{j=1}^n c_j=(s+1)i$.

Since \eqref{Bourbaki} is a Bourbaki sequence,  the module $\Coker \psi$ is torsionfree. Now, using the free presentation
$K_{i+1}  \xrightarrow{\psi} K_i/F \to \Coker\psi \to 0$ and $\Coker \psi \cong  Z_i/\partial_i(F)$, Theorem~\ref{a4.3} (a) yields $\height I_{s}(\psi)\ge 2$.

Arguing as in the first part of this proof, for all $k=1,\dots, n$ we find $\pb_1^{(k)},\dots, \pb_s^{(k)} \in \Ic_{i+1}$ and
$\qb_1^{(k)},\dots, \qb_s^{(k)}\in \Ic_i\setminus H$ such that  the vectors $\sum_{i=1}^s \mdeg(\pb_j^{(k)})$ and \\
$\sum_{i=1}^s \mdeg(\qb_j^{(k)})$  have the same $k$th component.

On one hand, the $k$th component of $\sum_{i=1}^s \mdeg(\pb_j^{(k)})$ is at least $$s-\left|\{ \pb \in \Ic_{i+1}\mid k\notin \pb\}\right|=s-\binom{n-1}{i+1}.$$ On the other hand, the $k$th component of $\sum_{i=1}^s \mdeg(\qb_j^{(k)})$ is at most $c_k$.
Thus $\sum_{j=1}^n c_j \geq  n\left(s-\binom{n-1}{i+1}\right)$, from where we infer that $i\geq (n-i)\binom{n-1}{i}-n\binom{n-1}{i+1}$.
\end{proof}

\begin{Corollary}
{\rm (a)} Let $i\ge 2$. Then there is no multigraded Bourbaki sequence of $Z_i$ for $n\gg0$.

{\rm (b)} Let $j\ge 3$. Then there is no multigraded Bourbaki sequence of $Z_{n-j}$ for $n\gg0$.
\end{Corollary}

\begin{proof}
(a) The polynomial $f(x)=i\binom{x-1}{i-1}-x\binom{x-1}{i-2}-i$  has degree $i-1$ and the coefficient of $x^{i-1}$ is $i/(i-1)! - 1/(i-2)!=1/(i-1)!>0$. It follows that $i\binom{n-1}{i-1}-n\binom{n-1}{i-2}-i=f(n)>0$ for all $n\gg0$, and we may apply
Theorem~\ref{a5.4}.

(b) The polynomial $g(x)= j\binom{x-1}{j-1}-x\binom{x-1}{j-2}-x+j$
has degree $j-1>1$ and the leading coefficient is $1/(j-1)!>0$. It follows that $\{n-(n-j)\}\binom{n-1}{n-j} - n\binom{n-1}{(n-j)+1}-(n-j) =g(n)>0$ for all $n\gg0$, and the conclusion follows from Theorem~\ref{a5.4} applied to $Z_{n-j}$.
\end{proof}

Here is one immediate application of Theorem~\ref{a5.4}.

\begin{Proposition}\label{reza}
\begin{enumerate}[{\rm (a)}]
\item If $n\geq 5$, there is no multigraded Bourbaki sequence of $Z_2$.
\item If $n\geq 8$, there is no multigraded Bourbaki sequence of $Z_{n-3}$.
\end{enumerate}
\end{Proposition}

\begin{proof}
It is easy to check  that $2<2 \binom{n-1}{1}-n\binom{n-1}{0}=n-2$, when $n\geq 5$, and that $n-3 < 3 \binom{n-1}{n-3}-n \binom{n-1}{n-2}$ when $n\geq 8$.  Then one applies Theorem~\ref{a5.4}.
\end{proof}

We formulate the following.

\begin{Question} For  $2\leq i \leq n-3$,  is there no multigraded Bourbaki sequence of $Z_i$?
\end{Question}

When $n=5$, the answer is positive, by Theorem~\ref{a5.4}. When $n=6$ the answer is also positive: the case $i=2$ is covered in Proposition~\ref{reza}, and the case $i=3$ is treated by ad-hoc methods in Proposition~\ref{n6z3}.

\subsection{Bourbaki ideals of $Z_{2}$ and $Z_{3}$}

In this section we construct a graded Bourbaki sequence and  determine explicitly a Bourbaki ideal of $Z_2$ for arbitrary $n$.

When $n=6$, we show in Proposition~\ref{n6z3} that $Z_3$ does not have a multigraded Bourbaki sequence. Nevertheless, we describe a graded one for it  in Proposition~\ref{n6z3-explicit}.

\begin{Proposition}\label{a5.5}
Let $F$ be the submodule of $K_2$ generated by the elements $\mathbf{e}_{i, i+1} - \mathbf{e}_{i+1, i+2}$ with  $1\le i\le n-2$ and let $\varphi=\partial_2|_{F}\: F\to Z_2$ be the restriction of $\partial_2$. Then $0\to F \xrightarrow{\varphi} Z_2 \to Z_2/\partial_2(F) \to 0$ is a graded Bourbaki sequence of $Z_2$.
\end{Proposition}

\begin{proof}
Set $\mathbf{f}_i=\mathbf{e}_{i, i+1} - \mathbf{e}_{i+1, i+2}$ for $1\le i\le n-2$. They are a free basis for $F$ and $\rank(F)=n-2=\rank(Z_2)-1$.
Let $\iota\: Z_2 \to K_1$ denote the inclusion map.
Let $A_n$ be the  matrix  representing  $\iota\circ \varphi$ with respect to the bases $\mathbf{f}_1,\dots, \mathbf{f}_{n-2}$ and $\mathbf{e}_1, \dots, \mathbf{e}_n$.
Then 
\[
A_n=\left(\begin{matrix}
-x_2 & & &&\\
x_1+x_3& -x_3& &&\mbox{\LARGE{0}}\\
-x_2& x_2 + x_4& \ddots&&\\
& -x_3& \ddots&\ddots&\\
& & \ddots&\ddots&-x_{n-1}\\
& & &\ddots&x_{n-2} + x_n\\
\mbox{\LARGE{0}}& & &&-x_{n-1}
\end{matrix}\right).
\]

We prove that $\height(I_{n-2} (A_{n}))\ge 2$ by induction on $n\ge 3$.
If $n=3$, then $I_{1} (A_{3})=(-x_2, x_1+x_3)$, whence its height is two. Assume our assertion holds for $n-1\geq 3$. Then
$A_n=\left(\begin{matrix}
A_{n-1} & *\\
0 & x_{n-1}
\end{matrix}\right)=
\left(\begin{matrix}
-x_2 & 0\\
* & B_{n-1}
\end{matrix}\right),$
where $B_{n-1}$ is the matrix obtained by replacing in $A_{n-1}$,   $x_i$ with $x_{i+1}$ for all $i$.
Hence
$$I_{n-2} (A_{n})\supseteq x_{n-1} I_{n-3} (A_{n-1}) + (-x_2)I_{n-3} (B_{n-1}).$$
Let $\frk{p}$ be any height one prime ideal of $S$. If $I_{n-2} (A_{n}) \subseteq \frk{p}$,
then $x_{n-1} I_{n-3}(A_{n-1}) \subseteq \fkp$.  By the induction hypothesis,  $\height (I_{n-3} (A_{n-1})) \geq 2$, hence $x_{n-1} \in \fkp$. Arguing similarly,  from $x_2 I_{n-3}(B_{n-1}) \subseteq \fkp$ we derive that $x_2\in \fkp$.
Then $(x_2, x_{n-1})\subseteq \frk{p}$, which is a contradiction. This shows that $\height(I_{n-2} (A_{n}))\ge 2$.

 Therefore, $0\to F \xrightarrow{\varphi} Z_2 \to Z_2/\partial_2(F) \to 0$ is a graded Bourbaki sequence of $Z_2$,  by Theorem \ref{a4.2}.
\end{proof}

Next we  compute  the Bourbaki ideal of $Z_2$ determined by the embedding  $\varphi$ in Proposition~\ref{a5.5}, keeping the notation from there.

Let $\psi: K_3\to K_2/F$ be the composition of $\partial_3$ and the canonical projection $K_2\to K_2/F$.
Arguing as in the proof of Theorem~\ref{a5.4}, we observe that
\[
K_3\xrightarrow{\psi} K_2/F\iso \bigoplus_{(i,j)\in H}S\mathbf{e}_{ij} \to Z_2/\partial_2(F) \to 0
\]
is a graded minimal  finite free presentation of $Z_2/ \partial_2(F)$, where we set
\[
H=\{(i,j) \mid 1\le i<j\le n \text{ and }  (i,j)\ne (2,3),(3,4)\dots, (n-1,n)\}.
\]

Let $B$ be the matrix representing $\psi$ with respect to  the canonical basis of $K_3$, and $\{ \mathbf{e}_{ij} \mid (i,j)\in H\}$, respectively. Let  $N=\rank(K_2/F)=\binom{n}{2}-(n-2)$. In order to apply Theorem~\ref{a4.4}, we will describe an $N\times (N-1)$ submatrix of $B$ of rank $N-1$.

For $1\le i\le n-2$ and $1\le j\le n-2$, let $C_{ij}$ be the $(n-i-1)\times (n-j-1)$ submatrix of $B$ with the  rows indexed by $\mathbf{e}_{i,i+2}, \dots, \mathbf{e}_{i,n}$ and the  columns indexed by $\mathbf{e}_{j,j+1,j+2}, \dots, \mathbf{e}_{j,j+1,n}$.

For $1\le j\le n-2$, let $D_j$ be the $1\times (n-j-1)$ submatrix of $B$ with obtained by selecting the row $\mathbf{e}_{12}$ and the columns $\mathbf{e}_{j,j+1,j+2}, \dots, \mathbf{e}_{j,j+1,n}$.

Then
\[
C=\left(
\begin{array}{ccc}
D_1&\cdots & D_{n-2}\\ \hline
C_{11}&\cdots &C_{1,n-2}\\
\vdots&\ddots&\vdots \\
C_{n-2,1}&\cdots &C_{n-2,n-2}
\end{array}
\right)
\]
is an $N\times (N-1)$ submatrix of $B$.

\begin{Lemma}\label{a5.6}
The following statements  hold.
\begin{enumerate}[{\rm (a)}]
\item For $i\ne j$, the entries of the first column of $C_{ij}$ are zero.
\item $C_{ij}=0$ for $1\leq i<j \leq n-2$.
\item $C_{ii}=\left(\begin{smallmatrix}
-x_{i+1} &  & \\
     &    \ddots  & \mbox{\LARGE{{\rm 0}}}    \\
\mbox{\LARGE{{\rm 0}}}    &    &    -x_{i+1}
\end{smallmatrix}\right) \text{ for } 1\leq i\leq n-2.$
\item $D_j=(x_j + x_{j+2}, x_{j+3}, x_{j+4}, \dots, x_n)$ for $1\le j\le n-2$.
\end{enumerate}
\end{Lemma}

\begin{proof}
(a) follows from the equation $\partial(\mathbf{e}_{j,j+1,j+2})= x_j \mathbf{e}_{j+1, j+2} - x_{j+1} \mathbf{e}_{j, j+2} + x_{j+2} \mathbf{e}_{j,j+1}$.

Parts (b), (c) and (d) follow from the equation
\[
\partial(\mathbf{e}_{j,j+1,q})= x_j \mathbf{e}_{j+1, q} - x_{j+1} \mathbf{e}_{j, q} + x_{q} \mathbf{e}_{j,j+1}
\]
for $1\le j\le n-2$ and $j+2\le q \le n$.
\end{proof}

With notation as above, there exist a graded ideal of height two isomorphic to $Z_2/\partial_2(F)$  and which we explicitly describe  as follows.

\begin{Theorem} A Bourbaki ideal of $Z_2$ is
$I=(1/a) I_{N-1} (C)$, where $a=\prod_{i=2}^{n-2} x_i^{n-1-i}$.
\end{Theorem}

\begin{proof}
By Lemma \ref{a5.6}, the matrix $C$ has the following form
\[
C=\left(
\begin{array}{cccc}
D_1&\cdots &\cdots & D_{n-2}\\ \hline
C_{11}&0 &\cdots &0\\
\vdots&\ddots&\ddots &\vdots \\
\vdots&&\ddots& 0 \\
C_{n-2,1}&\cdots &\cdots &C_{n-2,n-2}
\end{array}
\right),
\]
where
$C_{ii}=
\left(\begin{smallmatrix}
-x_{i+1} &  & \\
     &    \ddots  &     \\
   &    &    -x_{i+1}
\end{smallmatrix}\right)$
and
$C_{ij}=\left(
\begin{array}{c|ccc}
0 &&&\\
\vdots&&\mbox{\LARGE{*}}&\\
0&&\mbox{\LARGE{*}}&
\end{array}\right)$.
Hence $C$ has  rank $N-1$. It follows by Theorem \ref{a4.4} that there exists a unique element
$b\in S$ such that  $I=(1/b)I_{N-1} (C)$.

For $(i,j)\in H$, let $\Delta_{ij}({C})$ be the determinant of the matrix which is obtained from $C$ by deleting the row   corresponding to $\mathbf{e}_{ij}$.
Then $b=\gcd(\Delta_{ij} \mid  (i,j)\in H)$. Note that
for $1\le i\le n-2$, $\Delta_{i,i+2}({C})$ is the determinant of the matrix
\[
\left(
\begin{array}{ccc|c|ccc|ccc}
D_1&\cdots&D_{i-1}&x_i+x_{i+2}&x_{i+3}&\cdots&x_{n}&D_{i+1}&\cdots&D_{n-2}\\ \hline
C_{11} &&&&&&&&&\\
&\ddots&&\mbox{\Large{0}}&&\mbox{\Large{0}}&&&\mbox{\Large{0}}&\\
&&C_{i-1,i-1}&&&&&&&\\ \hline
&&&&-x_{i+1}&&&&&\\
&\mbox{\LARGE{*}}&&\mbox{\Large{0}}&&\ddots&&&\mbox{\Large{0}}&\\
&&&&&&-x_{i+1}&&&\\ \hline
&&&&&&&C_{i+1}&&\\
&\mbox{\LARGE{*}}&&\mbox{\Large{0}}&&\mbox{\LARGE{*}}&&&\ddots&\\\
&&&&&&&&&C_{n-2,n-2}
\end{array}
\right).
\]
Expanding this matrix with respect to the column corresponding to $\eb_{i,i+1,i+2}$, we see that $\Delta_{i,i+2}({C})=\pm (x_i+x_{i+2}) m/x_{i+1}$, where $m=\prod_{k=2}^{n-1} x_k^{n-k}$.
It follows that the greatest common divisor of $\Delta_{1,3}({C}), \dots, \Delta_{n-2,n}({C})$ is $a=\prod_{i=2}^{n-2} x_i^{n-1-i}$.

Clearly, $b$ divides $a$.
On the other hand, $I$ is generated in degree $n-2$, by Lemma~\ref{a5.1}(b). Hence
\[
n-2=N-1-\deg(b)\ge N-1-\deg(a)=\left\{\binom{n}{2}-(n-2)-1\right\} - \binom{n-2}{2}=n-2.
\]
It follows that $\deg(a)=\deg(b)$, thus $I=(1/a) I_{N-1} (C)$.
\end{proof}

\begin{Example}
We explain the previous constructions for $n=5$. Then
\[
C=\left(\begin{smallmatrix}
x_1+x_3 & x_4 &x_5&x_2+x_4&x_5&x_3+x_5\\
-x_2&0&0&0&0&0\\
0&-x_2&0&0&0&0\\
0&0&-x_2&0&0&0\\
0&x_1&0&-x_3&0&0\\
0&0&x_1&0&-x_3&0\\
0&0&0&0&x_2&-x_4
\end{smallmatrix}\right).
\]
After computing its maximal minors with CoCoA (\cite{Cocoa}), we find that a  Bourbaki ideal of $Z_2$ is
\begin{align*}
I=&(1/x_2^2x_3) I_{6} (C) \\
=&(x_2x_3x_4, x_1x_3x_4 + x_3^2x_4, x_1x_2x_4 + x_1x_4^2 + x_3x_4^2, x_1x_2x_3 + x_1x_2x_5 + x_1x_4x_5 + x_3x_4x_5, \\
&x_2^2x_4 + x_2x_4^2, x_2^2x_3 + x_2^2x_5 + x_2x_4x_5, x_2x_3^2 + x_2x_3x_5).
\end{align*}
\end{Example}


\begin{Proposition} \label{n6z3}
Suppose  $n=6$.
Let $F$ be a  free submodule of $K_3$   generated by   a subset of the  canonical basis of $K_3$ and let $\varphi=\partial_3|_{F}\: F\to Z_3$ be the restriction of $\partial_3$. Then the sequence
\begin{eqnarray}
\label{happy birthday}
0\to F \xrightarrow{\varphi} Z_3 \to Z_3/\partial_3(F) \to 0
\end{eqnarray}
is not a graded Bourbaki sequence of $Z_3$.
\end{Proposition}

\begin{proof}
Suppose that (\ref{happy birthday}) is a graded Bourbaki sequence of $Z_3$. Then $\rank(F)=\rank(Z_3)-1=9$, by  Lemma~\ref{a5.1}. Let  $\Bc$ be a basis of $F$ which is part of the canonical basis for $K_3$.

A key observation is that for any subset $\{m_1, m_2, m_3, m_4\}$ of $[6]:=\{1,2,\dots, 6\}$, at most two of 
\[
\mathbf{e}_{m_1, m_2, m_3}, \mathbf{e}_{m_1, m_2, m_4}, \mathbf{e}_{m_1, m_3, m_4}, \mathbf{e}_{m_2, m_3, m_4}
\]
are in $\Bc$. Indeed, if all four of them are in $\Bc$, then  $0\neq \partial_4  (\eb_{m_1, m_2, m_3, m_4}) \in F$, and $\varphi(\partial_4(\eb_{m_1, m_2, m_3, m_4}))=0$. This contradicts the injectivity of $\varphi$.
Now let us assume that among the former four elements, only $\mathbf{e}_{m_1, m_2, m_3}, \mathbf{e}_{m_1, m_2, m_4}, \mathbf{e}_{m_1, m_3, m_4}$ are in $\Bc$. 
Let $\ol{*}$ denote the residue class of an element $*$ of  $ K_3$    in  $K_3/(\Im \partial_4 + F)$.
 Then
\[
\overline{0}=\overline{\partial_4 (\mathbf{e}_{m_1, m_2, m_3, m_4})}=x_{m_1}\overline{\mathbf{e}_{m_2, m_3, m_4}}.
\]

If  $\overline{\eb_{m_2, m_3, m_4}}= \overline{0}$, then $\eb_{m_2, m_3, m_4} \in \Im \partial_4 +F$. Since $\Im \partial_4$ is generated in degree four, we get that $\eb_{m_2, m_3, m_4} \in F$, which is not the case. So, $\overline{\eb_{m_2, m_3, m_4}}$ is a nonzero torsion element in $K_3/(\Im \partial_4 + F)  \iso Z_3/\partial_3(F)$ which contradicts the fact that \eqref{happy birthday} is a Bourbaki sequence.

Thus, we conclude that for any distinct $m_1, m_2, m_3, m_4 \in [6]$,
$$
(\eb_{m_1, m_2, m_3} \in \Bc \text{ and } \eb_{m_1, m_2, m_4} \in \Bc)  \implies  (\eb_{m_2, m_3, m_4}\notin \Bc \text{ and } \eb_{m_1, m_3, m_4} \notin \Bc).
$$

 The $9$ subsets of $\{ 1, \dots, 6 \}$ which index the elements in $\Bc$ use 27 indices, so by the pigeon hole principle, there exists one index, say $6$, which is used at least 5 times.
Let $G$ be the graph on the vertex set $[5]$ and edges $E(G)=\{ (ij): \eb_{ij6} \in \Bc\}$.
Note that there is no cycle of length $3$ in $G$. Indeed, if $(ij), (jk), (ik) \in E(G)$, then $\eb_{ij6}, \eb_{jk6}, \eb_{ik6} \in \Bc$, which is false by the key observation above.

If $|E(G)| \geq 7$, then the complementary graph $\overline{G}$ has $5$ vertices and at most $3$ edges.
If  $j$ an isolated vertex  in $\overline{G}$, there exist  $i, k \in [5]\setminus\{j\}$ so that $(i,k)$ is not an edge in $\overline{G}$. This implies that $ijk$ is a 3-cycle in $G$, which is false.
In case $\overline{G}$ has no isolated vertex, up to relabeling, then   $E(\overline{G})=\{12, 13, 45\}$.
Thus $234$ is a 3-cycle in $G$, which is false.  
 
In case $|E(G)| =6$,  eventually indentifying first $\overline{G}$ which has $4$ edges, we remark that there are only six  possibilities for $G$ (up to a graph isomorphism). Among them only $G$ with edges $E(G)=\{12, 14, 23, 34, 25, 45\}$ has no cycle of length $3$.
To rule out also this possibility, we note that each pair of incident edges in $E(G)$ eliminates one possible element from $\Bc$.  
E.g. starting with the  edges (12), (14) we get that $\eb_{124}\notin \Bc$. Similarly, 
using the pairs $(12)$ and $(25)$, $(12)$ and $(23)$, $(14)$ and $(34)$ , $(14)$ and $(45)$, $(23)$ and $(25)$, $(23)$ and $(34)$, $(25)$ and $(45)$, respectively $(34)$ and $(45)$
one excludes $\eb_{125},\eb_{123},\eb_{134},\eb_{145},\eb_{235},\eb_{234},\eb_{245},\eb_{345} $, respectively.
The other four basis elements of $K_3$ containg $6$ are also not in $\Bc$. So far, from the $20$ elements of the canonical basis of $K_3$ we showed that  $13$ are not in $\Bc$, so $9=|\Bc| \leq 7$, a contradiction.

Therefore, $|E(G)|=5$.  Since $G$ has no $3$-cyle, after eventually relabeling the vertices we may assume  that $E(G)$  is either
\[
\{(1,2), (2,3), (3,4), (4,5), (1,5)\} \text{ or } \{ (1,2), (2,3), (3,4), (1,4), (1,5)\}.
\]

Assume the latter. The remaining four elements in $\Bc$ correspond to subsets of $[5]$ with three elements.
Arguing as above, the pairs of incident edges $(12)$ and $(15)$, $(12)$ and $(23)$, $(12)$ and $(14)$, $(23)$ and $(34)$, $(34)$ and $(14)$, $(14)$ and $(15)$ indicate that $\eb_{125},\eb_{123},\eb_{124},\eb_{234},\eb_{134},\eb_{145}$, respectively, are not in $\Bc$. So, the remaining four elements in $\Bc$ must be the four remaining ones $\eb_{135},\eb_{235},\eb_{245},\eb_{345}$. 
Since it is not possible to have $\eb_{235},\eb_{245}, \eb_{345}$ in $\Bc$ at the same time, we get a contradiction.

We are left with the case when 
  $\mathbf{e}_{126}, \mathbf{e}_{236}, \mathbf{e}_{346}, \mathbf{e}_{456}, \mathbf{e}_{156} \in \Bc$. Avoiding $3$-cyles in $G$ as before, we  infer that  the remaining four elements in $\Bc$ are  among
\[
\mathbf{e}_{124}, \mathbf{e}_{235}, \mathbf{e}_{134}, \mathbf{e}_{245}, \mathbf{e}_{135}.
\]

By symmetry, we may assume that $\mathbf{e}_{124}, \mathbf{e}_{235}, \mathbf{e}_{134}, \mathbf{e}_{245} \in \Bc$.  Let $\iota: Z_3 \to K_2$ be the inclusion map.
Then, by direct computation with CoCoA (\cite{Cocoa}), one can check that $I_9 (\iota \circ \varphi)\subseteq (x_2x_4x_6)$. It follows from Theorem~\ref{a4.2} that (\ref{happy birthday}) is not a Bourbaki sequence of $Z_3$.
\end{proof}

On the other hand, one can choose the basis of $F$ as follows.

\begin{Proposition} \label{n6z3-explicit}
Suppose  $n=6$.
Let $F$ be the submodule of $K_3$ generated by the elements
$\mathbf{e}_{124} - \mathbf{e}_{126}$,
$\mathbf{e}_{126} - \mathbf{e}_{134}$,
$\mathbf{e}_{134} - \mathbf{e}_{135}$,
$\mathbf{e}_{135} - \mathbf{e}_{156}$,
$\mathbf{e}_{156} - \mathbf{e}_{235}$,
$\mathbf{e}_{235} - \mathbf{e}_{236}$,
$\mathbf{e}_{236} - \mathbf{e}_{245}$,
$\mathbf{e}_{245} - \mathbf{e}_{346}$,
$\mathbf{e}_{346} - \mathbf{e}_{456}$.

Let $\varphi=\partial_3|_{F}\: F\to Z_3$ be the restriction of $\partial_3$. Then $0\to F \xrightarrow{\varphi} Z_3 \to Z_3/\partial_3(F) \to 0$ is a graded Bourbaki sequence of $Z_3$.

Moreover, $Z_3/\partial_3(F)\iso (1/x_1^4)I_{10}(C)(3)$, where
\[
C^{\mathrm{T}}=\left(\begin{smallmatrix}
  -x_2 + x_3& -x_4& 0& 0& 0& 0& x_1& 0& 0& 0& 0\\
  x_1 - x_2& -x_5& x_3& 0& 0& 0& 0& 0& 0& 0& 0\\
  x_1 + x_3& -x_6& 0& -x_2& 0& 0& 0& 0& 0& 0& 0\\
  x_1 - x_5& 0& x_4& 0& -x_2& 0& 0& 0& 0& 0& 0\\
  x_4 - x_6& 0& 0& 0& 0& -x_2& 0& x_1& 0& 0& 0\\
  -x_2 + x_5& 0& -x_6& 0& 0& 0& 0& 0& x_1& 0& 0\\
  x_4 - x_5& 0& 0& 0& -x_3& 0& 0& 0& 0& x_1& 0\\
  x_1 - x_6& 0& 0& x_4& 0& -x_3& 0& 0& 0& 0& 0\\
  -x_3 - x_6& 0& 0& x_5& 0& 0& 0& 0& 0& 0& x_1\\
  x_1 - x_4& 0& 0& 0& -x_6& x_5& 0& 0& 0& 0& 0
\end{smallmatrix}\right).
\]
\end{Proposition}

\begin{proof}
It is straightforward to check (with CoCoA \cite{Cocoa}) that $\height I_9 (\iota \circ \varphi) \ge 2$, where $\iota: Z_3 \to K_2$ is the inclusion map. It follows that $0\to F \xrightarrow{\varphi} Z_3 \to Z_3/\partial_3(F) \to 0$ is a graded Bourbaki sequence of $Z_3$, by Theorem \ref{a4.2}.
Let
\[
K_4\xrightarrow{\psi} K_3/F\iso \bigoplus_{(i,j,k)\in H}S\mathbf{e}_{ijk} \to I=Z_3/\partial_3(F) \to 0
\]
be a graded minimal  free presentation of $I=Z_3/\partial_3(F)$, where $\psi$ is the canonical composition and
{\small
\[
H=\{(1,2,4), (1,2,6),(1,3,4),(1,3,5),(1,5,6),(2,3,5),(2,3,6),(2,4,5),(3,4,6),(4,5,6)\}.
\]
}
If  $B$ is the matrix representing $\psi$ with respect to the canonical bases of $K_4$ and $\bigoplus_{(i,j,k)\in H}S\mathbf{e}_{ijk}$, each of these bases in the natural order, then
\[
{B}^{\mathrm{T}}=\left(\begin{smallmatrix}
  -x_2 + x_3& -x_4& 0& 0& 0& 0& x_1& 0& 0& 0& 0\\
  x_1 - x_2& -x_5& x_3& 0& 0& 0& 0& 0& 0& 0& 0\\
  x_1 + x_3& -x_6& 0& -x_2& 0& 0& 0& 0& 0& 0& 0\\
  x_1 - x_5& 0& x_4& 0& -x_2& 0& 0& 0& 0& 0& 0\\
  x_4 - x_6& 0& 0& 0& 0& -x_2& 0& x_1& 0& 0& 0\\
  -x_2 + x_5& 0& -x_6& 0& 0& 0& 0& 0& x_1& 0& 0\\
  x_4 - x_5& 0& 0& 0& -x_3& 0& 0& 0& 0& x_1& 0\\
  x_1 - x_6& 0& 0& x_4& 0& -x_3& 0& 0& 0& 0& 0\\
  -x_3 - x_6& 0& 0& x_5& 0& 0& 0& 0& 0& 0& x_1\\
  x_1 - x_4& 0& 0& 0& -x_6& x_5& 0& 0& 0& 0& 0\\
-x_3+x_4&0&0&0&0&0&-x_5&0&0&x_2&0\\
x_2+x_4&0&0&0&0&0&-x_6&-x_3&0&0&0\\
x_5-x_6&0&0&0&0&0&0&0&-x_3&0&x_2\\
x_2-x_6&0&0&0&0&0&0&x_5&-x_4&0&0\\
x_3+x_5&0&0&0&0&0&0&0&0&-x_6&-x_4
\end{smallmatrix}\right).
\]

Since $C$ is a submatrix of $B$ of full rank, there exists  a unique element $a\in S$ such that  $I=(1/a)I_{10}(C)$.
For $1\le s \le 11$, let $\Delta_{s}({C})$ be  the determinant of the matrix which is obtained by deleting the $s$th row of $C$.
Then, by computing (with CoCoA \cite{Cocoa}) the greatest common divisor of $\Delta_{1}({C}), \dots, \Delta_{11}({C})$, we see that $a=x_1^4$.
\end{proof}


\section{The Rees algebra of the Bourbaki ideal in Proposition~\ref{c5.3}}\label{section5}

In this section we consider
the Rees algebra of the Bourbaki ideal $I$ of $Z_{n-2}$ described  in Proposition~\ref{c5.3}. We show that it is a normal Cohen--Macaulay ring, and it is Gorenstein if $n$  is even. It turns out that the Rees algebra  of $Z_{n-2}$ has  the same properties as the Rees algebra of $I$, see \cite[Theorem 3.1 and Theorem 3.4]{SUV}.  At present we do not know whether our result can be directly deduced from \cite{SUV}.

Let, as before,  $S=K[x_1, x_2, \dots, x_n]$ be a polynomial ring of dimension $n\ge 2$ over a field $K$.
For an ideal $I$ of $S$, $\calR(I)=S[It]\subseteq S[t]$ is called  the {\it Rees algebra} of $I$, where $t$ denotes a variable over $S$.

\begin{Proposition}
\label{cm-normal}
Let $n\ge 3$ and $I$ be the Bourbaki ideal of $Z_{n-2}$ stated in Proposition~\ref{c5.3}. Then $\calR(I)$ is a Cohen-Macaulay normal domain.
\end{Proposition}

\begin{proof}
By Proposition \ref{c5.3},
\[
\calR(I)=K[x_1, \dots, x_n, \xb t/x_1x_2, \dots, \xb t/x_{n-1}x_n, \xb t/x_n x_1]\subseteq K[x_1, \dots,x_n, t],
\]
where $\xb=\prod_{i=1}^nx_i$. Therefore, to prove our assertion, it is enough to show that $\calR(I)$ is a normal toric ring, by Hochster's theorem (see \cite{Hochster} or \cite[Theorem 6.3.5]{BH}).

Let $\eb_1, \dots, \eb_n$ and  $\fb_1, \dots, \fb_n$  in $\ZZ^{n+1}$ be the exponent vectors of the monomials
\[
x_1, \dots, x_n \text{ and } \xb  t/x_1x_2, \dots, \xb t/x_{n-1}x_n, \xb t/x_n x_1,
\]
respectively. Set $C$ the affine semigroup generated by $\eb_1, \dots, \eb_n, \fb_1, \dots, \fb_n$. Note that given
$\ab=(a_1, \dots, a_{n+1})^{\mathrm{T}}\in \mathbb{Z}^{n+1}$, then
\begin{align}\label{normality}
\begin{split}
\ab \in C \Leftrightarrow & \ \ab=\sum_{i=1}^{n}r_i \eb_i + \sum_{j=1}^{n} s_j \fb_j \text{ for some nonnegative integers $r_i$ and $s_j$}\\
\Leftrightarrow & \ \ab=\left(\begin{smallmatrix}
r_1\\
r_2\\
r_3\\
\vdots\\
r_{n}\\
0
\end{smallmatrix}\right) +
\left(\begin{smallmatrix}
s-s_n-s_1\\
s-s_1-s_2\\
s-s_2-s_3\\
\vdots\\
s-s_{n-1}-s_n\\
s
\end{smallmatrix}\right)   \text{ where $r_i, s_j \geq 0$ are integers},
\end{split}
\end{align}
and $s=s_1+\cdots+s_n$.


Let  $D$ be the set of lattice points $\ab=(a_1, \dots, a_{n+1})^{\mathrm{T}}\in \mathbb{Z}^{n+1}$ in the rational cone which is obtained by intersecting the half spaces of equations:
\begin{align}\label{first}
a_1\ge 0,\ \dots,\ a_{n+1}\ge 0.
\end{align}
\begin{align}\label{second}
\text{ For $2\le \ell \le \lfloor n/2 \rfloor$, }  a_{i_1}+\dots+a_{i_{\ell}} \ge (\ell-1) a_{n+1},
\end{align}
where $1\le i_1<\dots<i_{\ell}\le n$ such that  $2\le i_{\ell'+1}-i_{\ell'}$ for $1\le \ell' \le \ell-1$ and $i_{\ell}-i_1\le n-2$.
\begin{align}\label{third}
a_1+\dots +a_n\ge (n-2)a_{n+1}.
\end{align}

By Gordan's lemma (see \cite[Proposition 6.1.2 (b)]{BH}) we obtain that $D$ is a normal affine semigroup.
  We prove that $C=D$, which implies that  $\calR(I)$ is a normal ring.
It is straightforward to check the inclusion $C\subseteq D$. Assume that $C\not\supseteq D$ and take an element $\ab=(a_1, \dots, a_{n+1})^{\mathrm{T}}\in D\setminus C$ so that $a_1+\dots +a_{n+1}$ is as small as possible. By the observation in (\ref{normality}), we have $a_{n+1}>0$.

\begin{Claim}\label{claim1}
$a_1>0, \dots, a_{n}>0$.
\end{Claim}

\begin{proof}[Proof of Claim \ref{claim1}]
Suppose that $a_{i_1}=a_{i_2}=0$ for some $1\le i_1 < i_2 \le n$. If $i_2-i_1\not=1\ \mod \ n$, by (\ref{second}), $a_{i_1}+a_{i_2}\ge a_{n+1}>0$, which is a contradiction. Hence $i_2-i_1=1\ \mod \ n$. By symmetry of $a_1, \dots, a_n$, we may assume that $i_1=1$ and $i_2=2$. If $n\geq 4$ then, by (\ref{second}), $a_i=a_1+a_i\ge a_{n+1}$ for $3\le i \le n-1$ and $a_n=a_2+a_n\ge a_{n+1}$.  If $n=3$, by \eqref{third} we also have $a_3\geq a_4$. Hence
\[
\ab=a_{n+1} \fb_1 + \left(\begin{smallmatrix}
0\\
0\\
a_3-a_{n+1}\\
\vdots\\
a_n-a_{n+1}\\
0
\end{smallmatrix}\right)\in C,
\]
which is a contradiction. Hence $0$ appears at most once among $a_1, \dots,a_n$ . Assume $a_1=0$. Then $a_2>0$ and we have
\[
a_i=a_1+a_i\ge a_{n+1} \text{ for $3\le i \le n-1$ and } a_2+a_n \ge a_{n+1}
\]
by (\ref{second})  and \eqref{third}. If $a_2 +a_n>a_{n+1}$, then, all of the inequalities (\ref{first}), (\ref{second}), and (\ref{third})  appearing $a_2$ are strict. It follows that $\ab -\eb_2\in D\setminus C$, which is a contradiction for the minimality of $\sum_{i=1}^{n+1} a_i$. Hence $a_2+a_n=a_{n+1}$. Then
\[
\ab=a_{2} \fb_n + a_n \fb_1 + \left(\begin{smallmatrix}
0\\
0\\
a_3-a_{n+1}\\
\vdots\\
a_{n-1}-a_{n+1}\\
0\\
0
\end{smallmatrix}\right)\in C.
\]
This is also a contradiction. Hence $a_1>0$. By the symmetry of $a_1, \dots, a_n$, we have $a_1>0, \dots, a_{n}>0$.
\end{proof}

Let us denote $a=\min\{a_1, \dots, a_n\}$ and $J=\{i \mid 1\le i\le n, a_i=a\}$.
If $|J|=n$, then $\ab=(a,\dots, a)^{\mathrm{T}}=a\fb_1+a\eb_1+a \eb_2 \in C$, which is false. Therefore, $|J|<n$.

We choose a subset
$J'=\{ j'_1<\dots <j'_v \} \subseteq J$
such that
\begin{equation}
\label{circular}
   2\le j'_{v'+1}-j'_{v'}\text{ for $1\le v'\le v-1$ and }  j'_v-j'_1\le n-2,
\end{equation}
 and $v$ is as large as possible.

For $1\leq w \leq n$ we define its (circular) predecessor to be $\pred(w)=w-1$, if $w>1$ and $\pred(1)=n$.
Similarly, its (circular) successor is $\succ(w)=w+1$, if $w<n$ and $\succ(n)=1$.

\begin{Claim}\label{claim2}
The set $J'$ above can be chosen  such that there exists $w$ in $J'$ with $\pred(w)\notin J'$.
\end{Claim}

\begin{proof}[Proof of Claim \ref{claim2}]
Assume that, for all $w\in J'$, $\pred(w) \in J$. Then we may replace $J'$ with the set $J''=\{\pred(w) | w\in J'\}$ which satisfies \eqref{circular} and $|J''|=|J'|$.
If $J''$ still does not have the desired property, we take predecessor sets until one finds a good substitute for $J'$. Indeed, this process must terminate in at most $n$ steps. Otherwise, it means that $|J|=n$, i.e. $\ab=(a,\dots, a)^{\mathrm{T}}$, which is false.
\end{proof}

We pick  $w \in J'$ so that $\pred(w) \notin  J$.

\begin{Claim}\label{claim3}
 $\ab-\fb_{w}\in D$.
\end{Claim}

\begin{proof} [Proof of Claim \ref{claim3}]
The vector $\ab-\fb_w$ satisfies \eqref{third}, and also \eqref{first} since $a_i>0$ for all $i$.
 We now verify \eqref{second}.

Assume $2\le \ell \le \lfloor n/2 \rfloor$ and $1\le i_1<\dots<i_{\ell}\le n$ such that  $2\le i_{\ell'+1}-i_{\ell'}$ for $1\le \ell' \le \ell-1$ and $i_{\ell}-i_1\le n-2$.

If $\{i_1,\dots, i_\ell\} \cap \{w, \succ(w)\} \neq \emptyset$, since  the $w$th and the $\succ(w)$-th components of $\fb_w$ are zero, it follows that the corresponding inequality \eqref{second} for $\ab-\fb_w$ and the  indices $i_1,\dots, i_\ell$ holds.

Assume $\{i_1,\dots, i_\ell\} \cap \{w, \succ(w)\} = \emptyset$. We then show that
$a_{i_1}+\dots+a_{i_\ell}>(\ell-1)a_{n+1}$, which implies that $a_{i_1}+\dots+a_{i_\ell}-\ell  \geq (\ell-1)(a_{n+1}-1)$ and that
\eqref{second} holds for $\ab-\fb_w$ and the indices $i_1,\dots, i_\ell$.

 We consider two cases.

If $\pred(w) \in \{i_1,\dots, i_\ell\}$, then since $a_{w}=a$ and $a_{\pred(w)}>a$, we obtain using \eqref{second} that
\[
a_{i_1}+\dots +a_{\pred(w)}+\dots + a_{i_\ell} > a_{i_1}+\dots +a_{w}+\dots +a_{i_\ell} \ge (\ell-1)a_{n+1}.
\]

Assume $\pred(w) \notin \{i_1,\dots, i_\ell\}$. Note that for all $1\le \ell' \le \ell$,
\[
a_{i_1}+\dots +a_{i_\ell'}+ \dots +a_{i_\ell}\ge a_{i_1}+ \dots +a_{w}+\dots +a_{i_\ell} \ge  (\ell-1)a_{n+1}.
\]
Hence, if $a_{i_1}+\dots+a_{i_\ell}=(\ell-1)a_{n+1}$, then $a_{i_1}=\dots=a_{i_\ell}=a$.
On the other hand,  using \eqref{second} for $J'$, we have $va \ge (v-1)a_{n+1}$.
Therefore, $va\ge (v-1)\ell a/(\ell-1)$, which implies $v\le \ell$. It follows that $\ell=v$ by the maximality of $v$.
However, the set $\{i_1, \dots, i_\ell \}\cup \{w\}\subseteq J$  also satisfies \eqref{circular}, which  contradicts
the maximality of  $|J'|$.
\end{proof}

Clearly, $\ab-\fb_w$ is not in the semigroup $C$, because otherwise $\ab\in C$, which is false. So, the vector $\ab-\fb_w\in D\setminus C$  has the sum of its components less than $\sum_{i=1}^{n+1} a_i$, which is false. Consequently, $C=D$ and
the normality of the Rees algebra $\calR(I)$ is now fully proven.
\end{proof}

\begin{Theorem}\label{good}
Let $\calR(I)$ denote the Rees algebra of the Bourbaki ideal $I$ of $Z_{n-2}$ stated in Proposition \ref{c5.3}. Then
\begin{enumerate}[{\rm (a)}]
\item If $n$ is even, then $\calR(I)$ is a Gorenstein normal domain.
\item If $n$ is odd, then $\calR(I)$ is a Cohen-Macaulay normal domain of type two.
\end{enumerate}
\end{Theorem}

\begin{proof}
Let $R=\calR(I)$.
In view of Proposition~\ref{cm-normal}, keeping the notation from its proof, $R$ is a toric ring generated by the monomials whose exponent is in the affine semigroup $C\subset\ZZ^{n+1}$, which satisfies $C=\RR_+ C \sect \ZZ^{n+1}$.

Let $\omega_R$ be  the ideal $(\xb^F| F\in \ZZ^{n+1} \cap  \relint \RR_+C)R$,  where $\relint \RR_+C$ denotes   the relative interior of the cone   $\mathbb{R}_+ C$.
Then $\omega_R$ is the canonical module of $R$ (cf.  \cite{Danilov}, \cite{Stanley}, or see \cite[Theorem 6.3.5(b)]{BH}).

Let $F \in \ZZ^{n+1}$. Then $F$ is in the semigroup $C$ if and only if its coordinates satisfy the weak inequalities
\eqref{first}, \eqref{second}, \eqref{third}; and moreover $F\in \relint \RR_+ C$ if and only if none of the latter inequalities becomes an equation.

This way, it is routine to check that  $F_1=(1, \dots, 1)^{\mathrm{T}}\in \mathbb{Z}^{n+1}$ is in $\relint \RR_+ C$.

Assume $F=(a_1, \dots, a_{n+1}) \in \relint \RR_+ C \cap \ZZ^{n+1}$.

(a) If $n$ is even, then $n=2k$. Clearly, $a_i -1 \geq 0$ for all $i=1,\dots, n+1$. Also, if $2\leq \ell \leq \lfloor n/2 \rfloor$ and $1\leq i_1 <\dots < i_\ell \leq n$ such that $2\leq i_{\ell'+1}-i_{\ell'}$ for $1\leq \ell' \leq \ell-1$ and $i_\ell- i_1 \leq n-2$,  then $a_{i_1}+\dots+a_{i_\ell} > (\ell-1) a_{n+1}$ implies that $(a_{i_1}-1)+\dots+(a_{i_\ell}-1) \geq (\ell-1) (a_{n+1}-1)$.

Using \eqref{second} twice, we obtain
\[
a_1+a_3 +\dots +a_{2k-1}> (k-1)a_{n+1} \text{ and } a_2+a_4 +\dots +a_{2k}> (k-1)a_{n+1}.
\]
Adding them yields $\sum_{i=1}^n (a_i-1) \geq (n-2) (a_{n+1}-1)$. Hence, the coordinates of $F-F_1$ satisfy the inequalities \eqref{first}, \eqref{second} and \eqref{third}, and $F-F_1 \in C$. Therefore, $\omega_R=(\xb^{F_1})R$ and $R$ is a Gorenstein ring.

(b) Suppose that $n=2k+1$. It is routine to check that $F_2=(k, \dots, k, k+1)^{\mathrm{T}} \in \relint \RR_+ C$, hence $(\xb^{F_1}, \xb^{F_2})R\subseteq \omega_R$.

We claim that $F-F_1 \in C$ or $F-F_2 \in  C$. This implies that $\omega_R  \subseteq (\xb^{F_1}, \xb^{F_2})R$.

Indeed, using \eqref{normality} we may write
\[
F=\sum_{i=1}^{n}r_i \eb_i + \sum_{j=1}^{n} s_j \fb_j =\left(\begin{smallmatrix}
r_1\\
r_2\\
r_3\\
\vdots\\
r_{n}\\
0
\end{smallmatrix}\right) +
\left(\begin{smallmatrix}
s-s_n-s_1\\
s-s_1-s_2\\
s-s_2-s_3\\
\vdots\\
s-s_{n-1}-s_n\\
s
\end{smallmatrix}\right),
\]
where $r_1, \dots, r_n$, $s_1, \dots, s_n$ are nonnegative integers and $s=s_1+\cdots+s_n$. Note that $r_1+\cdots+r_n>0$ since $F$ satisfies the strict inequality of (\ref{third}).

If  $r_1+\cdots+r_n\ge 2$, then $F-F_1$ satisfies the inequality (\ref{third}). It follows that $F-F_1\in C$.

In case   $r_1+\dots+r_n =1$, by symmetry, we may assume that $r_1=1$. Then
\[
\left(\begin{smallmatrix}
a_1\\
a_2\\
\vdots \\
a_n\\
a_{n+1}
\end{smallmatrix}\right)=F=
\left(\begin{smallmatrix}
1+s-s_n-s_1\\
s-s_1-s_2\\
\vdots\\
s-s_{n-1}-s_n\\
s
\end{smallmatrix}\right).
\]
We prove that $s_1>0, s_3>0, \dots, s_{2k+1}>0$.

Since $F$ satisfies the strict inequalities of (\ref{second}),  in particular we have that
\begin{align*}
a_2 + a_4+a_6 +\cdots +&a_{2k-4} + a_{2k-2} + a_{2k\phantom{+1}} > (k-1)a_{n+1},\\
a_2 + a_4+a_6 +\cdots +&a_{2k-4} + a_{2k-2} + a_{2k+1} > (k-1)a_{n+1},\\
a_2 + a_4+a_6 +\cdots +&a_{2k-4} + a_{2k-1} + a_{2k+1} > (k-1)a_{n+1},\\
&\vdots \\
a_2 + a_5+a_7 +\cdots +&a_{2k-3} + a_{2k-1} + a_{2k+1} > (k-1)a_{n+1},  \text{ and}\\
a_3 + a_5+a_7 +\cdots +&a_{2k-3} + a_{2k-1} + a_{2k+1} > (k-1)a_{n+1}.
\end{align*}
These inequalities imply that $ks-s+s_{2\ell+1}>(k-1)s$, that is, $s_{2\ell+1}>0$  for $0\le \ell\le k$.
On the other hand, since $F_2=\eb_1+\fb_1+\fb_3+\dots +\fb_{2k+1}$ we may write
\[
F-F_2=\sum_{\ell=0}^k (s_{2\ell+1}-1)\fb_\ell + \sum_{\ell'=1}^k s_{2\ell'}\fb_{\ell'} \in C,
\]
which proves our claim.

Note that $F_1-F_2$ has negative entries and $F_2-F_1$ does not satisfy \eqref{third}, so $F_1-F_2\notin C$ and $F_2-F_1\notin C$. Hence $(\xb^{F_1})R\neq (\xb^{F_1}, \xb^{F_2})R\neq (\xb^{F_2})R$. We conclude that when $n$ is odd, $\omega_R$ is minimally generated by $\xb^{F_1}$ and $\xb^{F_2}$, so the Cohen-Macaulay type of $R$ is two.
\end{proof}

\end{document}